\documentclass[12pt,oneside]{amsart}
\usepackage{amsmath,amssymb,latexsym,soul,cite,mathrsfs}
\usepackage{enumitem}
\usepackage{color,enumitem,graphicx}
\usepackage[colorlinks=true,urlcolor=blue,
citecolor=red,linkcolor=blue,linktocpage,pdfpagelabels,
bookmarksnumbered,bookmarksopen]{hyperref}
\usepackage[english]{babel}
\usepackage{lineno}

\usepackage[left=2.4cm,right=2.4cm,top=2.5cm,bottom=2.4cm]{geometry}

\usepackage[hyperpageref]{backref}

\numberwithin{equation}{section}
\newtheorem{theorem}{Theorem}[section]
\newtheorem{lemma}[theorem]{Lemma}
\newtheorem{corollary}[theorem]{Corollary}
\newtheorem{proposition}[theorem]{Proposition}

\newtheorem{claim}{Claim}
\newtheorem{definition}[theorem]{Definition}

\renewcommand{\epsilon}{\varepsilon}

\renewcommand{\rightarrow}{\to}

\title[Hardy-Sobolev type inequality with logarithmic term ]{On a supercritical Hardy-Sobolev type inequality with logarithmic term and related extremal problem}
\author[J.F.\ de Oliveira]{Jos\'{e} Francisco de Oliveira}\thanks{First author was partially supported by  CNPq grant number 309491/2021-5}
\address[J.F.\ de Oliveira]{
\newline\indent Department of Mathematics
	\newline\indent 
	Federal University of Piau\'{i}
	\newline\indent
	64049-550 Teresina, PI, Brazil}
	\email{\href{mailto:jfoliveira@ufpi.edu.br}{jfoliveira@ufpi.edu.br}}
	
\author[J.N.\ Silva ]{Jeferson Silva}
\address[J.N.\ Silva]{\newline\indent Department of Mathematics
	\newline\indent 
	Federal University of the Delta of Parna\'{i}ba
	\newline\indent
	CEP Parna\'{i}ba, PI, Brazil}
	\email{\href{mailto:j.n.silva@ufdpar.edu.br}{j.n.silva@ufdpar.edu.br}}
\begin{document}
\maketitle
\begin{abstract}
Our main goal  is to investigate  supercritical Hardy-Sobolev type inequalities with a logarithmic term and their corresponding  variational problem. We prove the existence of extremal functions for the associated variational problem, despite the loss of compactness. As an application, we show the existence of weak solution to a general class of  related elliptic partial differential equations with a logarithmic term.
\end{abstract}
\section{Introduction}
Let $B\subset\mathbb{R}^{N}$, $N \geq 3$ be the unit ball and denote by $H^{1}_{0,\mathrm{rad}}(B)$ the first order Sobolev space of radial functions.  The following  class of elliptic partial differential equations has recently been  investigated by several authors (cf. \cite{MR3514752, DCDS2019, Ngu, Deng-Peng-Zhang-Zhou,Deng-Zhang}  and the references therein)
\begin{equation}\label{m.000}
        \left\{
\begin{array}{lc}
 -\Delta \,u = g(x,u)|u|^{2^{*}-2} u, \quad u \in H^{1}_{0,\mathrm{rad}}(B) & \text{in} \quad B, \,\, \\
 u = 0, &\text{on} \quad \partial B ,   \\
\end{array}
\right.
    \end{equation}
 where $2^{*} = 2N/(N-2)$ is the critical Sobolev exponent and $g:\overline{B}\times\mathbb{R}\to\mathbb{R}$ is a  suitable function. In the purely critical case $g \equiv 1$, it is well known that problem \eqref{m.000} has no positive solution by  Pohozaev's identity. To overcome the nonexistence, in the celebrated paper \cite{MR0709644} Br\'{e}zis and Nirenberg proposed adding a lower-order perturbation term that makes it possible to avoid the loss of compactness arising from the critical growth. This type of problem is currently known  as Br\'{e}zis–Nirenberg problem, and there is a vast literature on this subject \cite{Band,SZ,CaoT,delP,Li}. Recently, J.M. do Ó et al. \cite{MR3514752}  proposed a new type of ``supercritical'' perturbation that plays a similar role of the Br\'{e}zis–Nirenberg lower-order perturbation term in  overcoming the loss of compactness. Namely, by choosing  $g(x,u) = u^{|x|^{\beta}},$ $u>0$  they were able to prove that  \eqref{m.000} admits at least one positive solution under the condition 
\begin{equation}\label{cond m.0}
    0<\beta < \min \{N/2, N-2\}.
\end{equation}
The approach used in \cite{MR3514752} also allowed them to prove the existence of extremal function to the variational problem
\begin{equation}\label{ORU1}
\sup \Big\{ \int_{B} |u(x)|^{2^{*}+|x|^{\beta}} dx \,: \,\, u \in H^{1}_{0, \mathrm{rad}}(B),\,\, \| \nabla u \|_{L^{2}(B)} =1 \Big\} 
\end{equation}
provided that $\beta$ satisfies the condition \eqref{cond m.0}. The supercritical problem \eqref{m.000}-\eqref{ORU1}  has at\-tracted the attention of several authors and there are many extensions to different contexts. In \cite{Cao}, the existence of nodal solutions for \eqref{m.000} was investigated. Extensions for the $k$-Hessian equation can be found in \cite{JDE2019, NARWA2021}, for Hardy-Sobolev type inequality in \cite{DCDS2019}, and for higher order derivative Sobolev spaces in \cite{Ngu}. For the analogous supercritical problem for Trudinger-Moser type growth, we recommend \cite{TMNgo,ORU2,ANNA} and the references therein.

In  \cite{Deng-Peng-Zhang-Zhou}, the authors investigated the problem \eqref{m.000}-\eqref{ORU1} with $g(x,u)=|\ln(\tau + |u|)|^{|x|^{\beta}}$ with $\tau\ge 1$ and were able to achieve results in line with those obtained in \cite{MR3514752}. Namely, they proved the following Sobolev inequality with a logarithmic term
\begin{equation}\label{m.1}
         \mathcal{F}_{\tau, \beta}=\sup \Big\{ \int_{B} |u|^{2^{*}}|\ln{(\tau+ |u(x)|)}|^{|x|^{\beta}} \, dx \,: \,\, u \in H^{1}_{0, \mathrm{rad}}(B),\,\, \| \nabla u \|_{L^{2}(B)} =1 \Big\} < \infty,
\end{equation}
where $\beta,\tau >0$ are real constants. In addition, the supremum in \eqref{m.1} is attained under the conditions \eqref{cond m.0} and $\tau\ge 1$. As an application, the authors are able to prove the existence of a positive solution for  \eqref{m.000} with supercritical logarithmic term, i.e, the equation
\begin{equation}\label{m.2}
        \left\{
\begin{array}{lc}
 -\Delta \,u = (\ln{(\tau +|u|)})^{|x|^{\beta}}|u|^{2^{*} - 2} u,  & \text{in} \quad B, \,\, \\
 u = 0, &\text{on} \quad \partial B ,   \\
\end{array}
\right.
    \end{equation}
admits a positive solution $u\in H^{1}_{0,\mathrm{rad}}(B)$ provided that \eqref{cond m.0}  holds and $1\leq \tau <\infty$. We also recommend \cite{Deng-Zhang} for an investigation into the existence of nodal solutions to the equation \eqref{m.2}.

Motivated by the works \cite{MR3514752,Deng-Peng-Zhang-Zhou,DCDS2019} and the classical Hardy inequality \cite{Hardy1920}, in this paper, we are interested in investigating Hardy-Sobolev type inequalities with a supercritical logarithmic term and their associated quasilinear elliptic  equations. In order to make our findings precise, we will first briefly introduce the weighted Sobolev spaces, which will be our working setting.  For $\theta\ge0$ and $q\ge 1$, set $L^q_{\theta}=L^q_{\theta}(0,R)$, $0<R\le\infty$ the Lebesgue space associated with the weighted measure $\mu=r^{\theta}dr$ on the interval $(0,R)$. Let us denote by $AC_{loc}(0,R)$ the set of all locally absolutely continuous functions on the interval $(0,R)$. Then, we consider the  Sobolev spaces
$$X^{1,p}_R=X^{1,p}_R\left(\alpha_0,\alpha_1\right)=cl\{u\in AC_{loc}(0,R)\;:\; \lim_{r\rightarrow R}u(r)=0,\;u\in L^{p}_{\alpha_0}\;\mbox{ and} \;u^{\prime}\in L^{p}_{\alpha_1} \}$$  with the norm $\|u\|_{X_R}=(\|u\|^{p}_{L^{p}_{\alpha_0}}+\|u^{\prime}\|^{p}_{L^{p}_{\alpha_1}})^{1/p}$.  In spite of its simple aspect,  mainly due to their connection with the classical Hardy inequality \cite{DCDS2019,Clement-deFigueiredo-Mitidieri,Hardy1920}, fractional dimension function spaces \cite{RMI2023,CV2023,PAMS2014}, and a general class of quasilinear elliptic operators including the $p$-Laplacian and the $k$-Hessian \cite{DLHJGA, DoLuHa,DCDS2019,JN-JF,NARWA2021, JDE2019}, the  weighted spaces $X^{1,p}_R$ has drawn the attention of several authors.
According to the relationship between the parameters $\alpha_1$ and $p$, we can distinguish three important cases for the spaces $X^{1,p}_R\left(\alpha_0,\alpha_1\right)$, namely, the Sobolev case $\alpha_1-p+1>0$, the Trudinger-Moser case $\alpha_1-p+1=0$, and the Morrey case  $\alpha_1-p+1<0$, see for instance \cite{DoLuHa, JN-JF}. In this paper, we are mainly interested in the bounded case $0<R<\infty$ submitted to the  Sobolev condition $\alpha_1-p+1>0$ and the transition condition $\alpha_0\ge\alpha_1-p$. In this case, the norm $\|u\|_{X_{R}}$ is equivalent to the gradient norm $\|u\|:=\|u^{\prime}\|_{L^{p}_{\alpha_1}}$ and  we have the  continuous embedding
\begin{equation}\label{eq10}
	X^{1,p}_R(\alpha_0,\alpha_1)\hookrightarrow L^{q}_{\theta},  \;\; \mbox{if}\;\;  1 < q\leq  p^{*} \;\;\mbox{and}\;\; \min\left\{\theta,\alpha{_0}\right\}\ge \alpha_1-p,
	\end{equation}
where 
$$p^*=\frac{(\theta+1)p}{\alpha_1-p+1}$$ 
represents  the critical (optimal) Sobolev exponent of $X^{1,p}_R(\alpha_0,\alpha_1)$. In addition, in the strict case $q<p^{*}$, the embedding \eqref{eq10} is compact.

From now on, we shall assume  the following conditions on the parameters $\alpha_0,\alpha_1, p$ and $\theta$:
\begin{equation}\label{cond m=1}
 p>1, \quad \alpha_1-p+1>0, \quad \alpha_0\ge \alpha_1-p \quad\text{and} \quad \theta>\alpha_1 -p.  
\end{equation}
The first part of this article is dedicated to the investigation of the supremum
\begin{equation}\label{phim.3}
         \mathcal{S}_{\theta,\tau}(\varphi)=\sup \Big\{ \int_{0}^{1} r^{\theta} |u(r)|^{p^{*}} |\ln{(\tau +|u(r)|)|^{\varphi(r)}} \, dr \,: \,\, u \in X^{1,p}_1(\alpha_0,\alpha_1), \,\, \|  u \|= 1 \Big\},
\end{equation}
where $\tau>0$ and  $\varphi:[0,1)\to\mathbb{R}$ is a continuous function satisfying the hypothesis
\begin{itemize}
    \item [($h_{1}$)]  $\varphi(0) = 0$ and $\varphi(r)>0$ for all $r \in(0, 1)$
    
    \item [($h_{2}$)] there exist  $c>0$ and $\sigma >1$ such that $\varphi(r)|\ln{r}|^{\sigma}\ln{|\ln{r}|} \leq  c$, for $r$ near $0$.
    
    \item [$(h_3)$] $\varphi(r)=o(|\ln(1-r)|)$,      as $r\to 1^{-}$, i.e. $\lim_{r\to 1^-}\varphi(r)(|\ln(1-r)|)^{-1}=0$.
\end{itemize}
Our first result is presented below.
\begin{theorem} \label{C3-T1}  Suppose that  $\varphi\in C[0,1)$ satisfies $(h_1)$-$(h_3)$. Then  $\mathcal{S}_{\theta,\tau}(\varphi)<\infty$ for any $\tau>0$. 
\end{theorem}
By choosing $\varphi(r)=r^{\beta}$, $\beta>0$ we can see that Theorem~\ref{C3-T1} improves and complements \eqref{m.1} since it includes both  $p\neq 2$ and non-integer values of $\theta$ satisfying \eqref{cond m=1}. In addition, it represents a counterpart of \cite[Theorem~1.1]{DCDS2019} with supercritical logarithmic term. 

To illustrate the scope of Theorem~\ref{C3-T1}, for $\tau\ge 1$, $a>1$ and $0\le b\le 1$ let us define $\Gamma_{a,b}:[0,\infty)\to\mathbb{R}$ given by $\Gamma_{a,b}(t)=t^{a}|\ln(\tau+t)|^{b}$ (if $b=0$ and $\tau=1$ we set $\Gamma_{a,b}(0)=0$). Then, the pair $(\Gamma_{a,b}, (0,1))$ is $\Delta
$-regular (see \cite{Adams}  and Lemma~\ref{lemma-convex} below) and we are able to consider the associated weighted Orlicz space  $L_{\theta,\Gamma_{a,b}}=L_{\theta,\Gamma_{a,b}}(0,1)$ endowed with the Luxemburg norm
\begin{equation}\label{LX-norm}
\|u\|_{\Gamma_{a,b}}=\inf\Big\{\lambda>0\;:\; \int_{0}^{1}r^{\theta}\Gamma_{a,b}\Big(\frac{|u(r)|}{\lambda}\Big)\,dr\le 1\Big\}.
\end{equation}
According with  \eqref{eq10}, the continuous  embedding $$X^{1,p}_1(\alpha_0,\alpha_1)\hookrightarrow L_{\theta,\Gamma_{p^*,0}}$$
is optimal, that is,  the Orlicz space $L_{\theta,\Gamma_{p^*,0}}$  is optimal to the above continuous embedding. In addition, if $0<b\le 1$, for  any $\delta>0$   we have $\Gamma_{p^*,0}(t)/\Gamma_{p^*,b}(\delta t)\to 0$ as $t\to\infty$  which means that  $\Gamma_{p^*,b}$ increases strictly more rapidly than $\Gamma_{p^*,0}$. Hence $X^{1,p}_1(\alpha_0,\alpha_1)$ may not be continuously embedded in any weighted Orlicz space $L_{\theta,\Gamma_{p^*,b}}$ with $0<b\le 1$. In spite of that, if we consider 
$b=b(r)=r^{\beta},\beta>0$ as a variable exponent, Theorem~\ref{C3-T1} yields the following:

\begin{corollary}   \label{C3-C1} Let $\tau\ge 1$ and $\beta>0$ be real numbers. Set 
$ L^{\beta}_{\theta, \ln}$  the generalized weighted Orlicz space  be given by
        $$L^{\beta}_{\theta, \ln} = \Big\{ u : (0, 1) \to \mathbb{R}\,\,\text{measurable}\; :\; \int_{0}^{1} r^{\theta} |u(r)|^{p^{*}} |\ln{(\tau +|u(r)|)}|^{r^{\beta}}\, dr < \infty\Big\},$$
endowed with the norm
\begin{equation}\label{CLXN-norm}
\|u\|_{L_{\theta, \ln}^{\beta}} = \inf \Big\{ \lambda > 0 :\,\, \int_{0}^{1} r^{\theta}\left|\dfrac{u(r)}{\lambda} \right|^{p^{*}} \left| \ln{\left(\tau + \left|\frac{u(r)}{\lambda}\right|\right)}\right|^{r^{\beta}} \, dr \leq 1 \Big\}.
\end{equation}
Then, we have the continuous embedding
    $X^{1,p}_1(\alpha_0,\alpha_1) \hookrightarrow L^{\beta}_{\theta, \ln}.$
\end{corollary}
Note that  $b(r)=r^{\beta}$ with $r\in [0,1)$ satisfies $0<b(r)\le 1$, except for  $r=0$. For this reason, we say that either the inequality with logarithmic term in Theorem~\ref{C3-T1} or the embedding in Corollary~\ref{C3-C1} are of the supercritical type.

In order to state our next results, we will consider the best Sobolev type constant 
\begin{equation}\label{C3-0}
    \Sigma_{p} =   \sup \Big\{ \int_{0}^{1} r^{\theta} |u|^{p^{*}}  \, dr \; : \; u \in X^{1,p}_1(\alpha_0,\alpha_1),\; \|  u \|= 1 \Big\}
\end{equation}
see for instance \cite{DCDS2019,Clement-deFigueiredo-Mitidieri} for more details. 

Our  next result  establishes the relation between  \eqref{phim.3} with $\varphi(r)=r^{\beta}$, $\beta>0$ and \eqref{C3-0} that  will allows us to analyze the attainability of \eqref{phim.3}. Firstly, let us denote   
\begin{equation}\label{m.3}
         \mathcal{F}_{\tau, \beta, \theta}= \mathcal{S}_{\theta,\tau}(r^{\beta})=\sup \Big\{ \int_{0}^{1} r^{\theta} |u|^{p^{*}} |\ln{(\tau +|u|)|^{r^{\beta}}} \, dr \,: \, u \in X^{1,p}_1(\alpha_0,\alpha_1), \; \|  u \|= 1 \Big\},
\end{equation}
where $\beta>0$.  With this notation,  we have the following
\begin{theorem}\label{T2}  For any $\tau> 0$, we have 
\begin{enumerate}
    \item [ $(i)$]  
$\mathcal{F}_{\tau, \beta, \theta} \geq \Sigma_{p}$ for any $\beta>0$
    \item[$(ii)$]  $ \mathcal{F}_{\tau, \beta, \theta} > \Sigma_{p}$ if $0<\beta< \min\{(\theta+1)/p, (\alpha_{1}-p+1)/(p-1)\}$
    \item[$(iii)$] $    \displaystyle\lim_{\beta \to \infty} \mathcal{F}_{\tau, \beta, \theta} = \Sigma_{p}$.
\end{enumerate}
\end{theorem}
Our attainability result reads below.
\begin{theorem} \label{C3-T3}
If $\tau \geq 1$ and $0<\beta< \min\{(\theta+1)/p, (\alpha_{1}-p+1)/(p-1)\}$, then the supremum $\mathcal{F}_{\tau, \beta, \theta}$ is attained.
\end{theorem}
Theorems \ref{T2} and \ref{C3-T3}  extend  results from \cite[Theorem~1.4]{Deng-Peng-Zhang-Zhou} to the context of the weighted Sobolev spaces $X^{1,p}_1(\alpha_0,\alpha_1)$. Related results can be found \cite{DCDS2019} for  $X^{1,p}_1(\alpha_0,\alpha_1)$ and \cite{NARWA2021} for the  $k$-admissible function spaces $\Phi^{k}_{0,rad}(B)$ which is the natural setting for studying the $k$-Hessian equation, see for instance \cite{MR3487276,Wangbook} for more details.

As a byproduct of the developments  in Theorem ~\ref{C3-T1},  Theorem~\ref{T2}, and Theorem~\ref{C3-T3}, we will investigate  the existence of solutions for the differential equation associated with the following  class of quasilinear elliptic operators in radial form
\begin{equation}\label{L-operator}
Lu\stackrel {\rm
def}{=}-r^{-\theta}(r^{\alpha_1}|u^{\prime}|^{p-2}u^{\prime})^{\prime}
\end{equation}
where  $\alpha_1, \theta\ge 0$ and $p>1$ are real numbers and $u\in C^{2}(0,R)$, with $0<R\le \infty$. The operator in \eqref{L-operator} can be applied to the study of microelectromechanical systems  MEMS (cf. \cite{Esposito,Esteban1}) and  has a close relationship with classical operators such as Laplace, $p$-Laplace  and $k$-Hessian, when acting on radially symmetric functions.  For an in-depth discussion on the class of operators \eqref{L-operator}, see for instance the papers \cite{Clement-deFigueiredo-Mitidieri, Jacobsen,deFigueiredo-Goncalves-Miyagaki, Esteban1,CX} and references quoted therein. Here, we shall prove the existence of solutions for the supercritical elliptic equation with logarithmic term associated with \eqref{L-operator}. 
\begin{theorem}    \label{C3-T4}
   Let $\alpha_0, \alpha_1, \theta$, and $p$ be positive real numbers satisfying \eqref{cond m=1}. For each $\tau \in [1, \infty)$ and $0< \beta< \min \{(\theta+1)/p, (\alpha_{1}-p+1)/(p-1)\}$, the problem 
   \begin{equation}\label{m.5}
        \left\{
\begin{array}{lc}
 Lu = (\ln(\tau + |u|))^{r^{\beta}}|u|^{p^{*}-2}u & \text{in} \quad (0, 1) \\
 u(1) = 0, &   \\
\end{array}
\right.
    \end{equation}
admits a  nontrivial  weak solution $u \in X^{1,p}_1(\alpha_0,\alpha_1)$.   
\end{theorem}
We observe that if we drop the logarithmic term in the equation \eqref{m.5}, the problem has no solution, as noted by \cite[Theorem~4.1]{Clement-deFigueiredo-Mitidieri}. Thus, the logarithmic term plays a similar role to the lower order term proposed by Brezis and Nirenberg \cite{MR0709644} (see also \cite{Clement-deFigueiredo-Mitidieri}), allowing us to avoid the levels of loss of compactness of the associated functional, see Lemma~\ref{lema B} and Lemma~\ref{lema A}.

The rest of this paper is organized as follows. In Section~\ref{sec2}, we prove the supercritical inequality with  logarithmic term in stated Theorem~\ref{C3-T1} and its consequence Corollary~\ref{C3-C1}. Section~\ref{sec2222} is devoted to prove of Theorem~\ref{T2}. In Section~\ref{sec4} we will prove the  result stated in Theorem~\ref{C3-T3}.   The proof of the existence result in Theorem~\ref{C3-T4} is given in Section~\ref{sec24}.
\section{Supercritical inequality: Proof of Theorem \ref{C3-T1}}\label{sec21}
\label{sec2}
In this section we will prove the supercritical Sobolev type inequality with logarithmic term stated in Theorem~\ref{C3-T1}. We   start recalling the following pointwise estimate result which can be found in \cite[Lemma 2.1]{DCDS2019}.
\begin{lemma}\label{C3-L1}
Let $0<R< \infty $ and assume $\alpha_{1} - p + 1> 0$.  Then, for any $u \in X^{1,p}_{R}(\alpha_0,\alpha_1)$ we have
    \begin{equation}\label{e8}
        |u(r)| \leq  \left[\frac{p-1}{\alpha_1-p+1}\left(1-\left(\frac{r}{R}\right)^{\frac{\alpha_1-p+1}{p-1}}\right)\right]^{\frac{p-1}{p}}\frac{\|u\|}{r^{\frac{\alpha_1-p+1}{p}}},\;\;  0<r\le R.
    \end{equation}
\end{lemma}
 \begin{proof}[Proof of the Theorem \ref{C3-T1}] 
Let $u \in X^{1,p}_{1}$ with $\| u\| = 1$.  Directly from \eqref{e8},  we can write
\begin{equation}\label{e13}
\begin{aligned}
  |u(r)| &\leq 
   \frac{\kappa}{r^{\frac{\alpha_1-p+1}{p}}}, \;\;  0<r\le 1
  \end{aligned}
\end{equation}
where $\kappa=
   \left((p-1)/(\alpha_1-p+1)\right)^{(p-1)/p}$.  In addition,  for any $\tau>0$, we have
\begin{equation}\label{assymp-log}
\lim_{t \to \infty} \dfrac{|\ln{\tau}|  + \ln{\left(1+\dfrac{\kappa t}{\tau}\right)}}{|\ln{\left(\tau +t\right)|}} = 1.
\end{equation}
Hence, there exists $C=C(\tau, \kappa)>0$ such that
\begin{equation}\label{assymp-log-2}
|\ln{\tau}| + \ln{\left(1+\frac{\kappa t}{\tau}\right)} \le C |\ln{\left(\tau +t\right)|},\;\; t\ge 1.
\end{equation}
 From  \eqref{e13}  and \eqref{assymp-log-2} we have
\begin{equation}\label{logu-logc}
\begin{aligned}
            |\ln {(\tau +|u|)}| 
            &\leq   |\ln{\tau}| + \ln{\Big(1+\frac{|u|}{\tau}\Big)}\\
            &\leq   |\ln{\tau}| + \ln{\Big(1+\frac{\kappa}{\tau} r^{-\frac{\alpha_{1}-p+1}{p}}\Big)}\\
            &\leq  C |\ln{\Big(\tau + r^{-\frac{\alpha_{1}-p+1}{p}}\Big)}|,
\end{aligned}
\end{equation}
for any $r\in (0,1]$.  Now, we write the split
\begin{equation}\label{T1-2.10}
    \int_{0}^{1} r^{\theta} |u|^{p^{*}} \left|\ln {(\tau +|u|)}\right|^{\varphi(r)}\, dr = I_1+I_2+I_3
\end{equation}
where 
\begin{equation}\nonumber
\begin{aligned}
I_1&=\int_{0}^{\rho} r^{\theta} |u|^{p^{*}} \left|\ln {(\tau +|u|)}\right|^{\varphi(r)}\, dr
 \\
 I_2&=\int_{\rho}^{\hat{\rho}} r^{\theta} |u|^{p^{*}} \left|\ln {(\tau +|u|)}\right|^{\varphi(r)} \, dr\\
 I_3&= \int_{\hat{\rho}}^{1} r^{\theta} |u|^{p^{*}} \left|\ln {(\tau +|u|)}\right|^{\varphi(r)} \, dr
\end{aligned}
\end{equation}
where $0<\rho<\hat{\rho}<1$ will  be chosen later.  From  \eqref{e13} and  \eqref{logu-logc} (c.f \eqref{C3-0}) we have
\begin{equation}\label{estimate01}
\begin{aligned}
    \int_{0}^{\rho} r^{\theta} |u|^{p^{*}} \left|\ln {(\tau +|u|)}\right|^{\varphi(r)}\, dr &\leq  \int_{0}^{\rho} r^{\theta} |u|^{p^{*}} \Big(\left|\ln {(\tau +|u|)}\right|^{\varphi(r)}-1\Big)\, dr + \Sigma_{p} \\
    &\leq  \int_{0}^{\rho} r^{\theta} |u|^{p^{*}}  \Big(\left|C \ln {\Big(\tau +r^{-\frac{\alpha_{1}-p+1}{p}}\Big)}\right|^{\varphi(r)}-1\Big)\, dr + \Sigma_{p} \\
     &= \int_{0}^{\rho}r^{\theta} |u|^{p^{*}}\Big(e^{\varphi(r) \ln{\left| C\ln \big(\tau +r^{-\frac{\alpha_{1}-p+1}{p}}\big)\right|}} - 1\Big)\, dr + \Sigma_{p}\\
    &\le  \int_{0}^{\rho}\frac{\kappa^{p^*}}{r} \Big(e^{\varphi(r) \ln{\left| C\ln \big(\tau +r^{-\frac{\alpha_{1}-p+1}{p}}\big)\right|}} - 1\Big)\, dr + \Sigma_{p},
    \end{aligned}
\end{equation}
for $\rho>0$ small enough. For any $r\in (0,\rho)$,  there exists $\vartheta=\vartheta(r)\in (0,1)$ such that 
$$
e^{\varphi(r) \ln{\big| C\ln \big(\tau +r^{-\frac{\alpha_{1}-p+1}{p}}\big)\big|}} - 1=e^{\varphi(r)\vartheta \ln{\big| C\ln \big(\tau +r^{-\frac{\alpha_{1}-p+1}{p}}\big)\big|}}\varphi(r) \ln{\big| C\ln \big(\tau +r^{-\frac{\alpha_{1}-p+1}{p}}\big)\big|}.
$$
By using  $(h_{2})$  we can see that 
$$\varphi(r) \ln{\big| C\ln \big(\tau +r^{-\frac{\alpha_{1}-p+1}{p}}\big)\big|} \leq c\dfrac{\ln{\big| C\ln \big(\tau +r^{-\frac{\alpha_{1}-p+1}{p}}\big)\big|}}{|\ln {r}|^{\sigma} \ln{|\ln{r}|}}, \quad \sigma>1$$
which implies 
$$e^{\varphi(r)\vartheta \ln{\big| C\ln \big(\tau +r^{-\frac{\alpha_{1}-p+1}{p}}\big)\big|}} \leq C,$$
for some $C>0$ large enough and $r$ near $0$.
Then, for $r>0$ small enough
\begin{equation}
\label{T1-2.7}
\begin{aligned}
     e^{\varphi(r) \ln{\big| C\ln \big(\tau +r^{-\frac{\alpha_{1}-p+1}{p}}\big)\big|}} - 1  &\le  C\varphi(r) \ln{\big| C\ln \big(\tau +r^{-\frac{\alpha_{1}-p+1}{p}}\big)\big|} \\
    &\leq  C \varphi(r) \ln {|\ln{r}}|\\
   &  \le  \frac{C_1}{|\ln r|^{\sigma}}, 
    \end{aligned}
\end{equation}
for some $C_1>0$, where in the last inequality we used the hypothesis $(h_2)$.  From \eqref{estimate01} and  \eqref{T1-2.7}, we have 
\begin{equation}\label{estimate01-01}
\begin{aligned}
   I_1= \int_{0}^{\rho} r^{\theta} |u|^{p^{*}} \left|\ln {(\tau +|u|)}\right|^{\varphi(r)}\, dr 
    & \le \kappa^{p^*}C_1 \int_{0}^{\rho} \frac{1}{r|\ln r|^{\sigma}}\, dr + \Sigma_{p}<\infty,
    \end{aligned}
\end{equation}
if $\rho>0$ is chosen small enough. In order to estimate $I_3$, we first observe that the function $r\mapsto \ln\big| C\ln \big(\tau +r^{-\frac{\alpha_{1}-p+1}{p}}\big)\big|$ is bounded near $1$. Thus, from  the assumption  $(h_3)$, there is $\hat{\rho}<1$ near $1$ such that  
\begin{equation}\label{e-hat}
\begin{aligned}
\varphi(r) \ln{\big| C\ln \big(\tau +r^{-\frac{\alpha_{1}-p+1}{p}}\big)\big|} &\le  \frac{1}{2}|\ln(1-r)|,\;\; \mbox{for all}\;\; r\in (\hat{\rho},1).
\end{aligned}
\end{equation}
By using  \eqref{e13}, \eqref{logu-logc} and \eqref{e-hat} we have
\begin{equation}\label{T1-2.13}
\begin{aligned}
I_3=\int_{\hat{\rho}}^{1} r^{\theta} |u|^{p^{*}} \left|\ln {(\tau +|u|)}\right|^{\varphi(r)} \, dr &\leq \int_{\hat{\rho}}^{1}\frac{\kappa^{p^*}}{r} \left|C\ln {(\tau + r^{-\frac{\alpha_{1}-p+1}{p}})}\right|^{\varphi(r)} \, dr \\
&= \int_{\hat{\rho}}^{1}\frac{\kappa^{p^*}}{r}e^{\varphi(r)\ln{\big| C\ln \big(\tau +r^{-\frac{\alpha_{1}-p+1}{p}}\big)\big|}} dr\\
&\le  \int_{\hat{\rho}}^{1}\frac{\kappa^{p^*}}{r(1-r)^{\frac{1}{2}}} dr\\
&=\kappa^{p^*}[2\ln((1-\hat{\rho})^{\frac{1}{2}}+1)-\ln\hat{\rho}].
\end{aligned}
\end{equation}
Finally,  \eqref{e13} and \eqref{logu-logc} yield
\begin{equation}\label{3-bounded}
\begin{aligned}
I_2=\int_{\rho}^ {\hat{\rho}} r^{\theta} |u|^{p^{*}} \left|\ln {(\tau +|u|)}\right|^{\varphi(r)} \, dr &\leq \int_{\rho}^ {\hat{\rho}}\frac{\kappa^{p^*}}{r} \left|C\ln {(\tau + r^{-\frac{\alpha_{1}-p+1}{p}})}\right|^{\varphi(r)} \, dr<\infty,
\end{aligned}
\end{equation}
where it is used that the function in the last integral is continuous on the compact interval $[\rho,\hat{\rho}]$. By combining \eqref{T1-2.10}, \eqref{estimate01-01}, \eqref{T1-2.13}, and \eqref{3-bounded}, we obtain the result.
 \end{proof}
 The next result ensures that  the generalized weighted Orlicz space  $L^{\beta}_{\theta, \ln}$ is well-defined, see for instance \cite{Die} for a more in-depth discussion on this topic.
 \begin{lemma}\label{lemma-convex} Let $a>1$, $0<b\le 1$ and $\tau\ge 1$ be real numbers. Set $\Gamma_{a,b}:[0,\infty)\to\mathbb{R}$ given by  $\Gamma_{a,b}(t)=t^{a}|\ln(\tau+t)|^{b}$. Then 
\begin{enumerate}
\item [$(a)$] $\Gamma_{a,b}$ is  continuous  on $[0,\infty)$
\item [$(b)$] $\Gamma_{a,b}$ is  convex on $[0,\infty)$
\item [$(c)$] $\lim_{t\to 0}\Gamma_{a,b}(t)/t=0$ and $\lim_{t\to \infty}\Gamma_{a,b}(t)/t=\infty$.
\end{enumerate} 
\end{lemma}
\begin{proof}
Given that $(a)$ and $(c)$ are clearly true, we will only prove $(b)$. First, we note  that 
$$t\Gamma_{a,b}(t)=\Gamma_{a+1, b}(t)\;\ \mbox{and}\;\;\Gamma_{a,b}(t)\ln(\tau+t)= \Gamma_{a,b+1}(t),\;\;  t>0.
$$
It is sufficient to show that $\Gamma^{\prime\prime}_{a,b}(t)>0$ for $t\in (0,\infty)$. Note that 
\begin{equation}\label{Gam1}
\Gamma^{\prime}_{a,b}(t)=a\Gamma_{a-1,b}(t)+\frac{b}{\tau+t}\Gamma_{a,b-1}(t).
\end{equation}
 Let $g(t)=\frac{1}{\tau+t}\Gamma_{a,b-1}(t)$, with $t>0$. From \eqref{Gam1}, we can write
\begin{equation}\label{g-part}
\begin{aligned}
g^{\prime}(t)& =-\frac{1}{(\tau+t)^2}\Gamma_{a,b-1}(t)+\frac{1}{\tau+t}\big[a\Gamma_{a-1,b-1}(t)+\frac{b-1}{\tau+t}\Gamma_{a,b-2}(t)\big]\\
& =-\frac{1}{(\tau+t)^2}\Gamma_{a,b-1}(t)+\frac{1}{\tau+t}\Big[\frac{a}{t}\Gamma_{a,b-1}(t)+\frac{b-1}{\tau+t}\frac{1}{\ln(\tau+t)}\Gamma_{a,b-1}(t)\Big]\\
&=\frac{\Gamma_{a,b-1}(t)}{(\tau+t)^2\ln(\tau+t)}\Big[-\ln(\tau+t)+a(\tau+t)\frac{\ln(\tau+t)}{t}+b-1\Big]\\
&=\frac{\Gamma_{a,b-1}(t)}{(\tau+t)^2\ln(\tau+t)}\Big[(a-1)\ln(\tau+t)+a\tau\frac{\ln(\tau+t)}{t}+b-1\Big].
\end{aligned}
\end{equation}
On the other hand, 
\begin{equation}\label{Gam2}
\begin{aligned}
\Gamma^{\prime}_{a-1,b}(t)&=(a-1)\Gamma_{a-2,b}(t)+\frac{b}{\tau+t}\Gamma_{a-1,b-1}(t)\\
&=\frac{a-1}{t^2}\ln(\tau+t)\Gamma_{a,b-1}(t)+\frac{b}{(\tau+t)t}\Gamma_{a,b-1}(t)\\
&=\frac{\Gamma_{a,b-1}(t)}{(\tau+t)^2\ln(\tau+t)}\Big[(a-1)\Big((\tau+t)\frac{\ln(\tau+t)}{t}\Big)^2+ b\Big((\tau+t)\frac{\ln(\tau+t)}{t}\Big)\Big].
\end{aligned}
\end{equation}
By setting $h_{\tau}(t)=(\tau+t)\frac{\ln(\tau+t)}{t}$, $t>0$ and using \eqref{Gam1}, \eqref{g-part} and \eqref{Gam2}, we can write
\begin{equation}\label{Gam-grande}
\begin{aligned}
& \Gamma^{\prime\prime}_{a,b}(t)=a\Gamma^{\prime}_{a-1,b}(t)+bg^{\prime}(t)\\
&=\frac{\Gamma_{a,b-1}(t)}{(\tau+t)^2\ln(\tau+t)}\Big[a(a-1)[h_{\tau}(t)]^2+abh_{\tau}(t)+b(b-1)\Big]\\
&+ \frac{\Gamma_{a,b-1}(t)}{(\tau+t)^2\ln(\tau+t)}\Big[b(a-1)\ln(\tau+t)+ab\tau\frac{\ln(\tau+t)}{t}\Big].
\end{aligned}
\end{equation}
 From \eqref{Gam-grande}, we only need to show that 
\begin{equation}\label{Phi-G}
\Phi_{\tau}(t)=a(a-1)[h_{\tau}(t)]^2+b[ah_{\tau}(t)+b-1], \;\; t> 0
\end{equation}
is a positive function. We note that for any $\tau\ge 1$, we have $h_{\tau}(t)\ge h_1(t)=(1+t)\frac{\ln(1+t)}{t}$. In addition,   $\lim_{t\to 0^+}h_1(t)=1$ and 
\begin{equation}
\begin{aligned}
h^{\prime}_1(t)
&=\frac{1}{t}\Big[1-\frac{\ln(1+t)}{t}\Big]>0\\
\end{aligned}
\end{equation}
for any $t>0$. Hence, $h_1$  is an increasing function on $(0,\infty)$ and, for any $\tau\ge 1$ we have $h_{\tau}(t)\ge h_1(t)\ge 1 $, for all $t\in (0,\infty)$. It follows that 
\begin{equation}\nonumber
\begin{aligned}
\Phi_{\tau}(t)&\ge a(a-1)+b[a+b-1]>0
\end{aligned}
\end{equation}
which completes the proof.
\end{proof}
 
\begin{proof}[Proof of the Corollary~\ref{C3-C1}] First of all, we will see that $L^{\beta}_{\theta,\ln}$  is a Banach space. Let us define  $\psi:(0,1)\times [0,\infty)\to [0,\infty)$ given by $\psi(r, t)= t^{p^*} |\ln(\tau + t)|^{r^{\beta}}$, with $\beta>0$. From Lemma~\ref{lemma-convex},  $\psi$ is a generalized $\Phi$-function (cf. \cite[Definition~2.3.9]{Die}), that is,
\begin{enumerate}
\item [$(i)$] for every $r\in (0,1)$, the function $g:[0,\infty)\to[0,\infty)$ given by $g(t)=\psi(r,t)$ is  convex, left-continuous, $g(0)=\lim_{t\to 0^+}g(t)=0$, and $\lim_{t\to\infty}g(t)=\infty$.
\item [$(ii)$] $r\mapsto \psi(r, t)$ is a Lebesgue mensurable function for any $t\ge 0$.
\end{enumerate}
Now, let us consider the  semimodular induced by $\psi$, that is, $\varrho_{\ln}:L^{0}_{\theta}\to [0,\infty]$ given by
\begin{equation}\nonumber
\varrho_{\ln}(u)= \int_{0}^{1} r^{\theta} |u |^{p^{*}} \left| \ln{(\tau + |u|)}\right|^{r^{\beta}} \, dr,
\end{equation}
where $L^0_{\theta}=L^0((0,1),r^{\theta}dr)$ denotes  the space of all Lebesgue measurable functions on $(0,1)$.  Since $\psi$ is a generalized $\Phi$-function, from \cite[Theorem~2.3.13]{Die} we have that the generalized Orlicz space $L^{\beta}_{\theta,\ln}$  is a  Banach space  endowed  with the  Luxemburg norm
\begin{equation}\label{LXN}
\|u\|_{L^{\beta}_{\theta,\ln}}=\inf\Big\{\lambda>0\;:\;\varrho_{\ln}\Big(\frac{u}{\lambda}\Big)\le 1\Big\}.
\end{equation}  
Now, for any $u\in X^{1,p}_1(\alpha_0,\alpha_1)\setminus\{0\}$ and $\lambda_0>1$ large such that $\lambda^{p^*}_0\ge \mathcal{F}_{\tau, \beta, \theta}$,   Theorem~\ref{C3-T1} implies
\begin{equation}
\begin{aligned}
\varrho_{\ln}\Big(\frac{u}{\lambda_0\|u\|}\Big)&= \frac{1}{\lambda_{0}^{p^*}}\int_{0}^{1} r^{\theta} \Big|\frac{u}{\|u\|}\Big|^{p^{*}} \Big| \ln\Big(\tau + \frac{1}{\lambda_0}\Big|\frac{u}{\|u\|}\Big|\Big)\Big|^{r^{\beta}} \, dr\\
&\le \frac{1}{\lambda^{p^*}_{0}}\int_{0}^{1} r^{\theta} \Big|\frac{u}{\|u\|}\Big|^{p^{*}} \Big| \ln\Big(\tau + \Big|\frac{u}{\|u\|}\Big|\Big)\Big|^{r^{\beta}} \, dr\\
&\le  \frac{1}{\lambda^{p^*}_0}\mathcal{F}_{\tau, \beta, \theta}\le 1.
\end{aligned}
\end{equation}
It follows that  
\begin{equation}\label{LxEmb}
\|u\|_{L^{\beta}_{\theta,\ln}}\le \lambda_0\|u\|,\;\;\mbox{for all} \;\; u\in X^{1,p}_{1}(\alpha_0,\alpha_1)
\end{equation}   
which proves the result.
\end{proof}
\section{Sharp estimates: Proof Theorem \ref{T2}}\label{sec2222}
This section is devoted to prove the sharp estimates stated in Theorem \ref{T2}. The proof is based on the modified \textit{Bliss function} introduced by \cite{Clement-deFigueiredo-Mitidieri} and follows some ideas in \cite{MR0709644} and \cite{Deng-Peng-Zhang-Zhou,DCDS2019}. Firstly, for each $0< R \leq \infty$, let
\begin{equation}\label{sp*-constant}
    \mathcal{S}(p^{*}, R) = \inf \left\{ \int_{0}^{R} r^{\alpha_{1}} |u^{\prime}|^{p} \,dr :\, u \in X^{1, p}_{R}(\alpha_0, \alpha_1) \quad \text{and} \quad \|u\|_{L^{p^{*}}_{\theta}}=1 \right\}.
\end{equation}
It is known that $\mathcal{S}(p^{*}, R)$ is independent of $R$ and that it is achieved only in the case $R= + \infty$. In addition, for each $\epsilon>0$, if we set
\begin{equation}\label{funcaoextremal}
    u_{\epsilon}^{*}(r) =  \dfrac{ \widehat{c} \epsilon^{s}}{\left( \epsilon^{n} + r^{n}\right)^{\frac{1}{m}}}, \quad r \geq 0
\end{equation}
where
\begin{equation}\label{relations-paramenters}
\left\{
\begin{array}{lll}
s= \dfrac{\alpha_{1}-p+1}{p^2-p} \\
n = \dfrac{\theta-\alpha_{1}+p}{p-1}\\
m = \dfrac{\theta-\alpha_{1}+p}{\alpha_{1}-p+1}\\
\widehat{c} = \left[(\theta+1)\left(\dfrac{\alpha_{1}-p+1}{p-1}\right)^{p-1}\right]^{\frac{1}{p}\frac{\alpha_{1}-p+1}{\theta-\alpha_{1}+p}}
\end{array}
\right.
\end{equation}
then
\begin{equation}\label{3.4}
    \mathcal{S}^{\frac{\theta+1}{\theta-\alpha_{1}+p}} = \int_{0}^{\infty} r^{\theta} |u^{*}_{\epsilon}|^{p^{*}}\, dr = \int_{0}^{\infty} r^{\alpha_{1}} |(u^{*}_{\epsilon})'|^{p}\, dr,
\end{equation}
where $\mathcal{S}$ denotes the value common of $\mathcal{S}(p^{*}, R)$ for all $ R \in (0, \infty]$,  see for instance \cite[Proposition~1.4]{Clement-deFigueiredo-Mitidieri}  and \cite{DCDS2019} for more details. Since $\mathcal{S}(p^{*}, R)$ is indepentent of $R$, by using \eqref{C3-0},  \eqref{sp*-constant} and \eqref{3.4} we get the identity
\begin{equation}\label{S-E}
\Sigma_p=\mathcal{S}^{-\frac{p^*}{p}}=\mathcal{S}^{-\frac{\theta+1}{\alpha_{1}-p+1}}.
\end{equation}
 Let us take a suitable cut-off   $\eta \in C^{\infty}_{0}(0, 1)$ such that  $0\le \eta\le 1$ on $[0,1]$ and 
\begin{equation}\label{cut-off}
\eta(r) \equiv 1\;\;\mbox{on}\;\; (0,r_0]\;\;\mbox{and}\;\;  \eta(r) \equiv 0\;\;\mbox{on}\;\; [2r_0,1 ]
\end{equation}
for some $0< r_{0} < 2r_{0}<1$. Then, according to   \cite[Claim~1]{DCDS2019} we have
\begin{equation}\label{3.6}
    \|(\eta u^{*}_{\epsilon})'\|^{p}_{L_{\alpha_{1}}^{p}} = \mathcal{S}^{\frac{\theta+1}{\theta-\alpha_{1}+p}} + O(\epsilon^{sp}), \quad \text{as}\,\, \epsilon \to 0
\end{equation}
and 
\begin{equation}\label{3.7}
    \|\eta u^{*}_{\epsilon}\|^{p^{*}}_{L_{\theta}^{p^{*}}} = \mathcal{S}^{\frac{\theta+1}{\theta-\alpha_{1}+p}} + O(\epsilon^{sp^{*}}), \quad \text{as}\,\, \epsilon \to 0.
\end{equation}
For $\widehat{A}>0$, $\beta >0$ and $A = \widehat{A} \widehat{c}$ we define 
\begin{equation}\label{definition u}
    u_{\epsilon}(r) = \widehat{A} \eta(r) u^{*}_{\epsilon}(r) = A \eta (r)\dfrac{ \epsilon^{s}}{\left( \epsilon^{n} + r^{n}\right)^{\frac{1}{m}}}.
\end{equation}
In order to get a precise estimate for the \textit{bubbles} functions $u_{\epsilon}$ along as the supercritical logarithmic functional, for  any $t>0$ and $0\le a<b\le1$, we introduce the notation
\begin{equation}\label{EtAB}
\mathcal{E}_{t}(a,b)= \int_{a}^{b} r^{\theta} |u_{\epsilon}|^{p^{*}}\Big( |\ln{(\tau + t|u_{\epsilon}|)}|^{r^{\beta}}-1\Big) \,dr.
\end{equation}
Let us also highlight some useful relations on the parameters in \eqref{relations-paramenters}.
\begin{equation}\label{relations-paramenters-two}
\left\{
\begin{array}{lll}
\dfrac{sm}{n}= \dfrac{1}{p} \\
n - sm= \dfrac{\theta-\alpha_{1}+p}{p}\\
sp^{*} = \dfrac{\theta+1}{p-1}\\
sp=\dfrac{\alpha_1-p+1}{p-1}\\
\theta - \dfrac{np^{*}}{m} + 1 = - \dfrac{\theta+1}{p-1}\\
s-\dfrac{n}{m} =- \dfrac{\alpha_{1}-p+1}{p}\\
\left(s-\dfrac{n}{m} \right)p^{*} = -(\theta+1)
\end{array}
\right.
\end{equation}
These identities will appear throughout our calculations, and we will use them whenever necessary without further comments.
\begin{lemma}\label{lema20} Let $(t_{\epsilon})$ be  any positive real sequence such that $\displaystyle\lim_{\epsilon\to 0}t_{\epsilon}=t_0>0$.
Then, there exists a constant $C>0$ such that
\begin{equation}\label{Lemma-estimative}
       \mathcal{E}_{t_{\epsilon}}(0,1)\ge   
        \left\{
\begin{array}{lll}
        C \epsilon^{\beta} \ln {|\ln {\epsilon}|} + O(\epsilon^{\frac{\theta+1}{p}}), & \hbox{if} &  0<\tau < e \\
         C \epsilon^{\beta} \ln {|\ln {\epsilon}|} , & \hbox{if} &   e \leq \tau < \infty,
\end{array}
\right.
\end{equation}
for $\epsilon>0$ small sufficiently.
\end{lemma}
\begin{proof}
Firstly,  if $\epsilon>0$ is small enough, $\eta\equiv1$ on $(0,\epsilon)$, and thus  $t_{\epsilon}u_{\epsilon}(r)\ge d \epsilon^{-\frac{\theta+1}{p^*}}\ge e$ for all $r\in (0,\epsilon)$, where $d=t_0A/2^{1+1/m}$. Thus, $\ln(\tau+ t_{\epsilon}|u_{\epsilon}|)\ge\ln(\tau+ d \epsilon^{-\frac{\theta+1}{p^*}})\ge 1$ and $\ln{|\ln{(\tau+ t_{\epsilon}|u_{\epsilon}|)|}}\ge 0$ on $ (0, \epsilon)$ and since $e^{t} \geq 1 + t$, for $t\in\mathbb{R}$, we can write
\begin{equation}\label{3.15}
\begin{aligned}
\mathcal{E}_{t_{\epsilon}}(0,\epsilon)& = \int_{0}^{\epsilon} r^{\theta} |u_{\epsilon}|^{p^{*}} \left( |\ln{(\tau+ t_{\epsilon}|u_{\epsilon}|)|^{r^{\beta}} - 1}\right)\, dr\\
    &\geq \Big(\frac{d}{t_{\epsilon}}\Big)^{p^{*}} \epsilon^{-(\theta+1)}\int_{0}^{\epsilon} r^{\theta} \left( e^{r^{\beta}\ln{ |\ln{(\tau+ d\epsilon^{-\frac{\theta+1}{p^*}})|}}} - 1\right)\, dr  \\
    &\geq \Big(\frac{d}{t_{\epsilon}}\Big)^{p^{*}} \epsilon^{-(\theta+1)}\int_{0}^{\epsilon} r^{\theta+ \beta}\ln{ |\ln(\tau+ d\epsilon^{-\frac{\theta+1}{p^*}})|}\, dr\\ 
    &\geq  C \epsilon^{-(\theta+1)} \ln {|\ln{\epsilon}|} \int_{0}^{\epsilon} r^{\theta + \beta}\, dr\\
    &\geq C \epsilon^{\beta} \ln {|\ln{\epsilon}|},
    \end{aligned}
    \end{equation}
for suitable $C>0$ and $\epsilon>0$ small enough. Now, we divide our argument into two cases. 
\paragraph{\textbf{Case~1}} $0<\tau < e$. \\
Firstly, we consider the sub-case $0<\tau<1/e$.  We observe that  \eqref{definition u} and $\widehat{A}=A/\hat{c}$ imply  
\begin{equation}\nonumber
\begin{aligned}
 \big| \ln{(\tau + t_{\epsilon}\widehat{A} u^{*}_{\epsilon})}\big| \leq 1\;\;\mbox{iff}\;\; e^{-1}-\tau\le \frac{t_{\epsilon}A\epsilon^{s}}{(\epsilon^n+r^{n})^{\frac{1}{m}}}\le e-\tau.
\end{aligned}
\end{equation}
Then, by setting
\begin{equation}\label{3.11-a}
    a_{\epsilon}= \left[ \left(\dfrac{t_{\epsilon}A \epsilon^{s}}{e-\tau}\right)^{m} - \epsilon^{n} \right]^{\frac{1}{n}} = \epsilon^{\frac{sm}{n}} \left[ \left(\dfrac{t_{\epsilon}A}{e-\tau}\right)^{m} - \epsilon^{n-sm} \right]^{\frac{1}{n}}
\end{equation}
and
\begin{equation}\label{b-est}
    b_{\epsilon}= \left[ \left(\dfrac{t_{\epsilon}A \epsilon^{s}}{e^{-1}-\tau}\right)^{m} - \epsilon^{n} \right]^{\frac{1}{n}} = \epsilon^{\frac{sm}{n}} \left[ \left(\dfrac{t_{\epsilon}A}{e^{-1}-\tau}\right)^{m} - \epsilon^{n-sm} \right]^{\frac{1}{n}}
\end{equation}
we can see that 
\begin{equation}\label{star-log}
    \big| \ln{(\tau + t_{\epsilon}\widehat{A} u^{*}_{\epsilon})}\big| \leq 1 \quad \text{if and only if} \quad a_{\epsilon}\le r\le b_{\epsilon}.
\end{equation}
Note that $b_{\epsilon}> a_{\epsilon}>0$ and, from \eqref{relations-paramenters-two}, $a_{\epsilon}=O(\epsilon^{\frac{1}{p}})$ and $b_{\epsilon}=O(\epsilon^{\frac{1}{p}})$, as $\epsilon \to 0$. In particular, for $\epsilon>0$ small enough we have
\begin{equation}\label{eabr}
0<\epsilon<a_{\epsilon}<b_{\epsilon}<r_0<1
\end{equation}
where $r_0$ is given in \eqref{cut-off}. Now, according to the partition \eqref{eabr}, we write the split 
\begin{equation}\label{split-six}
 \mathcal{E}_{t_{\epsilon}}(0,1)=\mathcal{E}_{t_{\epsilon}}(0,\epsilon)+\mathcal{E}_{t_{\epsilon}}(\epsilon,a_{\epsilon})+\mathcal{E}_{t_{\epsilon}}(a_{\epsilon},b_{\epsilon})+\mathcal{E}_{t_{\epsilon}}(b_{\epsilon},r_0)+\mathcal{E}_{t_{\epsilon}}(r_0,1).
\end{equation}
Since $u_{\epsilon}=\widehat{A}u^{*}_{\epsilon}$ on $(\epsilon, a_\epsilon)$, \eqref{star-log} implies
\begin{equation}\label{I2}
\begin{aligned}
\mathcal{E}_{t_{\epsilon}}(\epsilon,a_{\epsilon})= \int_{\epsilon}^{a_{\epsilon}} r^{\theta} |u_{\epsilon}|^{p^{*}} \left( |\ln{(\tau+ t_{\epsilon}|u_{\epsilon}|)|^{r^{\beta}} - 1}\right)\, dr\ge 0.
    \end{aligned}
    \end{equation}
By using \eqref{relations-paramenters-two}, we have
\begin{equation}\label{atob}
\begin{aligned}
  0& \le   \int_{a_{\epsilon}}^{b_{\epsilon}} r^{\theta} \dfrac{\epsilon^{sp^{*}}}{\left( \epsilon^{n} + r^{n}\right)^{\frac{p^*}{m}}} \, dr   \le  \epsilon^{sp^*}\int_{a_{\epsilon}}^{b_{\epsilon}} r^{\theta-\frac{np^*}{m}}\, dr   \\
  &=\frac{p-1}{\theta+1}\epsilon^{\frac{\theta+1}{p-1}}\left(a^{-\frac{\theta+1}{p-1}}_{\epsilon}-b^{-\frac{\theta+1}{p-1}}_{\epsilon}\right)\\
  &=\epsilon^{\frac{\theta+1}{p}}\frac{p-1}{\theta+1}\left[\left(\frac{\epsilon^{\frac{1}{p}}}{a_{\epsilon}}\right)^{\frac{\theta+1}{p-1}}-\left(\frac{\epsilon^{\frac{1}{p}}}{b_{\epsilon}}\right)^{\frac{\theta+1}{p-1}}\right]=O(\epsilon^{\frac{\theta+1}{p}}).
 \end{aligned}
\end{equation}
Since  $0\le u_{\epsilon}=\widehat{A}u^{*}_{\epsilon}$ on $(a_{\epsilon}, b_\epsilon)$, \eqref{atob} yields
\begin{equation}\label{u-3.14}
\begin{aligned}
0 &\leq \int_{a_{\epsilon}}^{b_{\epsilon}} r^{\theta}|u_{\epsilon}|^{p^{*}} \, dr \le  A^{p^*}
    \int_{a_{\epsilon}}^{b_{\epsilon}} r^{\theta}\dfrac{\epsilon^{sp^*}}{\left( \epsilon^{n} + r^{n}\right)^{\frac{p^*}{m}}} \, dr =O(\epsilon^{\frac{\theta+1}{p}}).
\end{aligned}
\end{equation}
So,  from \eqref{star-log}  and \eqref{u-3.14}
\begin{equation}\label{I3}
\begin{aligned}
 \mathcal{E}_{t_{\epsilon}}(a_{\epsilon},b_{\epsilon})&= \int_{a_{\epsilon}}^{b_{\epsilon}} r^{\theta} |u_{\epsilon}|^{p^{*}} \left( |\ln{(\tau+ t_{\epsilon}|u_{\epsilon}|)|^{r^{\beta}} - 1}\right)\, dr\\
 &=\int_{a_{\epsilon}}^{b_{\epsilon}} r^{\theta} |u_{\epsilon}|^{p^{*}}  |\ln(\tau+ t_{\epsilon}|u_{\epsilon}|)|^{r^{\beta}}\, dr -\int_{a_{\epsilon}}^{b_{\epsilon}} r^{\theta} |u_{\epsilon}|^{p^{*}}\, dr
 &=O(\epsilon^{\frac{\theta+1}{p}}).
    \end{aligned}
    \end{equation}
On the interval $(b_{\epsilon}, r_0)$ we also have $0\le u_{\epsilon}=\widehat{A}u^{*}_{\epsilon}$. Thus, \eqref{star-log} implies
\begin{equation}\label{I4}
\begin{aligned}
 \mathcal{E}_{t_{\epsilon}}(b_{\epsilon},r_0)= \int_{b_{\epsilon}}^{r_0} r^{\theta} |u_{\epsilon}|^{p^{*}} \left( |\ln{(\tau+ t_{\epsilon}|u_{\epsilon}|)|^{r^{\beta}} - 1}\right)\, dr\ge 0.
    \end{aligned}
    \end{equation}
 For any $r\in (r_0, 1)$, we have the estimate
 \begin{equation}\nonumber
 0\le t_{\epsilon}u_{\epsilon}(r)=t_{\epsilon}A \eta (r)\epsilon^{s}\left( \epsilon^{n} + r^{n}\right)^{-\frac{1}{m}}\le t_{\epsilon}A\epsilon^{s}\left( \epsilon^{n} + r^{n}_0\right)^{-\frac{1}{m}}\le 2t_0A \epsilon^{s}r^{-\frac{n}{m}}_0,
 \end{equation}
which, together with the assumption $0<\tau<1/e$,  implies  that $\tau+t_{\epsilon}|u_{\epsilon}|\le 1/e$ on $(r_0, 1)$, if $\epsilon>0$ is small enough. Thus, $|\ln(\tau+t_{\epsilon}|u_{\epsilon}|)|\ge 1$ on $(r_0, 1)$. It follows that 
\begin{equation}\label{I5}
\begin{aligned}
 \mathcal{E}_{t_{\epsilon}}(r_0,1)= \int_{r_0}^{1} r^{\theta} |u_{\epsilon}|^{p^{*}} \left( |\ln{(\tau+ t_{\epsilon}|u_{\epsilon}|)|^{r^{\beta}} - 1}\right)\, dr\ge 0.
    \end{aligned}
    \end{equation}
By combining  \eqref{3.15}, \eqref{split-six}, \eqref{I2}, \eqref{I3}, \eqref{I4} and \eqref{I5},  we obtain
\begin{equation}\label{3.16}
\begin{aligned}
\mathcal{E}_{t_{\epsilon}}(0,1)
&\ge C \epsilon^{\beta} \ln{|\ln{\epsilon}|} + O(\epsilon^{\frac{\theta+1}{p}}).
\end{aligned}
\end{equation}
Next, we treat the sub-case $1/e\leq \tau < e$. Here,  there holds
\begin{equation}\label{case2<1}
\big| \ln{(\tau + t_{\epsilon}\widehat{A}u^{*}_{\epsilon})}\big| \leq 1 \quad \text{if and only if} \quad a_{\epsilon} < r< 1
\end{equation}
where $a_{\epsilon}$ is defined in \eqref{3.11-a}. In this case, by choosing $\epsilon>0$ small enough such that $0<\epsilon<a_{\epsilon}<r_0<1 $, we consider the split
\begin{equation}\label{split-2subcaso}
\mathcal{E}_{t_{\epsilon}}(0,1)=\mathcal{E}_{t_{\epsilon}}(0,a_{\epsilon})+\mathcal{E}_{t_{\epsilon}}(a_{\epsilon},1).
\end{equation} 
Since   $u_{\epsilon}=\widehat{A}u^{*}_{\epsilon}$ on $(0, a_{\epsilon})$, from \eqref{case2<1}  and  \eqref{3.15} we can write
\begin{equation}\label{case2-I1}
\begin{aligned}
\mathcal{E}_{t_{\epsilon}}(0,a_{\epsilon})=\mathcal{E}_{t_{\epsilon}}(0,\epsilon)+\mathcal{E}_{t_{\epsilon}}(\epsilon,a_{\epsilon})\ge \mathcal{E}_{t_{\epsilon}}(0,\epsilon)
\ge C \epsilon^{\beta} \ln {|\ln{\epsilon}|},
\end{aligned}
\end{equation}
for suitable $C>0$. Analogous to \eqref{atob}, we have 
\begin{equation}\label{atob-2}
\begin{aligned}
  0& \le   \int_{a_{\epsilon}}^{1} r^{\theta} \dfrac{\epsilon^{sp^{*}}}{\left( \epsilon^{n} + r^{n}\right)^{\frac{p^*}{m}}} \, dr   \le  \epsilon^{sp^*}\int_{a_{\epsilon}}^{1} r^{\theta-\frac{np^*}{m}}\, dr   \\
  &=\frac{p-1}{\theta+1}\epsilon^{\frac{\theta+1}{p-1}}\left(a^{-\frac{\theta+1}{p-1}}_{\epsilon}-1\right)\\
  &=\epsilon^{\frac{\theta+1}{p}}\frac{p-1}{\theta+1}\left[\left(\frac{\epsilon^{\frac{1}{p}}}{a_{\epsilon}}\right)^{\frac{\theta+1}{p-1}}-\epsilon^{\frac{\theta+1}{p^2-p}}\right]=O(\epsilon^{\frac{\theta+1}{p}}).
 \end{aligned}
\end{equation}
From \eqref{atob-2} we have 
\begin{equation}\label{u2-3.14}
\begin{aligned}
0 &\le \int_{a_{\epsilon}}^{1} r^{\theta}|u_{\epsilon}|^{p^{*}} \, dr \le  A^{p^*}
    \int_{a_{\epsilon}}^{1} r^{\theta}\dfrac{\epsilon^{sp^*}}{\left( \epsilon^{n} + r^{n}\right)^{\frac{p^*}{m}}} \, dr =O(\epsilon^{\frac{\theta+1}{p}}).
\end{aligned}
\end{equation}
Now, for any $r\in (a_{\epsilon}, 1)$,   we have 
\begin{equation}\label{Ea1}
\begin{aligned}
 0\le  t_{\epsilon} u_{\epsilon}(r) \le\Bar{d}\epsilon^{s}\left( \epsilon^{n} + a^n_{\epsilon}\right)^{-\frac{1}{m}} &= \Bar{d}\epsilon^{s-\frac{n}{mp}}\left(\frac{\epsilon^{\frac{1}{p}}}{a_{\epsilon}}\right)^{\frac{n}{m}}\left(1+ \Big(\frac{\epsilon}{a_{\epsilon}}\Big)^{n}\right)^{-\frac{1}{m}}\\
 &=\Bar{d}\left(\frac{\epsilon^{\frac{1}{p}}}{a_{\epsilon}}\right)^{\frac{n}{m}}\left(1+ \Big(\frac{\epsilon}{a_{\epsilon}}\Big)^{n}\right)^{-\frac{1}{m}}.
 \end{aligned}
\end{equation}
where $\Bar{d}=2t_0A$, if $\epsilon>0$ is small enough.  Since $a_{\epsilon}=O(\epsilon^{\frac{1}{p}})$, as $\epsilon \to 0$ we have that  $\ln\tau \le \ln(\tau+t_{\epsilon}u_{\epsilon}))\le s_{\tau}$ on $(a_{\epsilon},1)$ for some positive constant $s_{\tau}$ depending only on $\tau$.  Hence, $| \ln(\tau+t_{\epsilon}u_{\epsilon}))|^{r^{\beta}}\le c$ on $(a_{\epsilon},1)$, for some $c>0$ depending only on $\beta$ and $\tau$.  Thus, from \eqref{u2-3.14} we have
\begin{equation}\label{2-3.14}
\mathcal{E}_{t_{\epsilon}}(a_{\epsilon}, 1)=\int_{a_{\epsilon}}^{1} r^{\theta}|u_{\epsilon}|^{p^{*}}   \left| \ln{(\tau + t_{\epsilon}| u_{\epsilon}|)}\right|^{r^{\beta}} \, dr  - \int_{a_{\epsilon}}^{1} r^{\theta}|u_{\epsilon}|^{p^{*}} \, dr= O(\epsilon^{\frac{\theta+1}{p}}).
\end{equation}
From \eqref{split-2subcaso}, \eqref{case2-I1} and  \eqref{2-3.14}, it follows that 
\begin{equation}\label{3.19}
\begin{aligned}
 \mathcal{E}_{t_{\epsilon}}(0,1)\ge   C \epsilon^{\beta} \ln{|\ln{\epsilon}|} + O(\epsilon^{\frac{\theta+1}{p}}).
\end{aligned}
\end{equation}
By \eqref{3.16} and \eqref{3.19}, we get \eqref{Lemma-estimative} for the case $0< \tau < e$.
\paragraph{\textbf{Case~2}} $e \leq \tau < \infty$.\\
 In this case, we have
$\left|\ln(\tau+ t_{\epsilon}u_{\epsilon})\right| \ge  1 \;\; \mbox{on} \;\;  (0,1).$
Hence,  from \eqref{3.15} we have
\begin{equation}\nonumber
\begin{aligned}
\mathcal{E}_{t_{\epsilon}}(0,1)\ge \mathcal{E}_{t_{\epsilon}}(0,\epsilon)\ge C \epsilon^{\beta} \ln{|\ln{\epsilon}|}
\end{aligned}
\end{equation}
for $\epsilon>0$ small enough. This proves \eqref{Lemma-estimative} for $e\leq \tau < \infty$.
\end{proof}
Next, we will provide an upper estimate for $\mathcal{E}_{t_{\epsilon}}(0,1)$. Namely,
\begin{lemma}\label{lema20-B}  Let $\beta >0$ and  $(t_{\epsilon})$  as in Lemma~\ref{lema20}. Then, as $\epsilon\to 0$ we have
 \begin{enumerate}
 \item [$(a)$] If $0<\tau<e$, 
 \begin{equation}\label{Caso A -final}
 \mathcal{E}_{t_{\epsilon}}(0,1)\le  \left\{\begin{aligned}
    &O(\epsilon^{\beta}|\ln\epsilon|\ln(|\ln {\epsilon}|))+O(\epsilon^{\frac{\theta+1}{p}})\;&\mbox{if}&\; \beta=\frac{\theta+1}{p-1}\\
    &O(\epsilon^{\frac{\theta+1+\beta}{p}}\ln(|\ln {\epsilon}|))+O(\epsilon^{\beta}\ln(|\ln {\epsilon}|)) +O(\epsilon^{\frac{\theta+1}{p-1}}) \;&\mbox{if}&\; \beta\not=\frac{\theta+1}{p-1}.
\end{aligned}\right.
\end{equation} 
 \item [$(b)$]  If  $e \leq \tau < \infty$,
 \begin{equation}\label{Caso B-final}
\mathcal{E}_{t_{\epsilon}}(0,1)\le \left\{\begin{aligned}
 &O(\epsilon^{\beta}|\ln\epsilon|\ln(|\ln\epsilon|)) &\mbox{if}&\;\; \beta=\frac{\theta+1}{p-1}\\
&O(\epsilon^{\frac{\theta+\beta+1}{p}}\ln(|\ln\epsilon|))+O(\epsilon^{\beta}\ln(|\ln\epsilon|))+O(\epsilon^{\frac{\theta+1}{p-1}})&\mbox{if}& \;\;\beta\not=\frac{\theta+1}{p-1}.
\end{aligned}\right.
\end{equation}
 \end{enumerate}
\end{lemma}
\begin{proof}
Firstly, note that 
\begin{equation}\label{uppper-ue}
 0\le  t_{\epsilon} u_{\epsilon}(r) =  t_{\epsilon} A \eta (r)\epsilon^{s}\left( \epsilon^{n} + r^{n}\right)^{-\frac{1}{m}}\le \Bar{d}\epsilon^{s-\frac{n}{m}},  \;\; \mbox{for all}\;\; r\in (0,1).
\end{equation}
where $\Bar{d}=2t_0A$, if $\epsilon>0$ is small enough. 
\paragraph{\textbf{Case~A}} $0<\tau < e$.\\
Let $a_{\epsilon}$ and $b_{\epsilon}$ be given by  \eqref{3.11-a} and \eqref{b-est}, for $\epsilon>0$ small enough such that \eqref{eabr} holds. By definition, for all $r\in (0, a_{\epsilon})$, we have $\ln(\tau+t_{\epsilon}u_{\epsilon})\ge 1$. Hence, by using \eqref{uppper-ue} we obtain
\begin{equation}\label{log-logX}
|\ln(\tau+t_{\epsilon}u_{\epsilon})|=\ln(\tau+t_{\epsilon}u_{\epsilon})\le \ln(\tau+\Bar{d}\epsilon^{s-\frac{n}{m}})=|\ln\epsilon|\Big(\frac{n}{m}-s+o_{\epsilon}(1)\Big)\le \xi|\ln\epsilon|\;\; \mbox{on}\;\; (0,a_{\epsilon})
\end{equation}
for some $\xi>0$.  Thus,  $0\le r^{\beta}\ln (|\ln \epsilon|)\le a_{\epsilon}^{\beta}\ln (|\ln \epsilon|)\to 0$ as $\epsilon\to 0$, for  $r\in (0, a_{\epsilon})$. Hence, 
by using that $(e^x-1)/x\to 1$ as $x\to 0$, we can write
 \begin{equation}\label{L1<}
 \begin{aligned}
& \mathcal{E}_{t_{\epsilon}}(0,a_{\epsilon})= \int_{0}^{a_{\epsilon}} r^{\theta} |u_{\epsilon}|^{p^{*}} \left(e^{r^{\beta}\ln|\ln(\tau + t_{\epsilon} |u_{\epsilon}|)|} -1\right)\, dr\\
    &\le  \int_{0}^{a_{\epsilon}} r^{\theta} |u_{\epsilon}|^{p^{*}} \left(e^{r^{\beta}\ln(\xi|\ln \epsilon|)} -1\right)\, dr\\
   & \leq c_1 \ln(\xi|\ln {\epsilon}|)\int_{0}^{a_{\epsilon}} r^{\theta+ \beta}|u_{\epsilon}|^{p^*}\, dr \\
    &= c_1A^{p^*} \ln(\xi|\ln {\epsilon}|) \int_{0}^{a_{\epsilon}} r^{\theta+ \beta}\dfrac{\epsilon^{sp^{*}}}{\left( \epsilon^{n} + r^{n}\right)^{\frac{p^*}{m}}}\, dr \\
&\leq  c_1A^{p^*} \ln(\xi|\ln {\epsilon}|) \Big[ \epsilon^{(s-\frac{n}{m})p^*} \int_{0}^{\epsilon} r^{\theta+ \beta} \, dr + \epsilon^{sp^*} \int_{\epsilon}^{a_{\epsilon}} r^{\theta+ \beta - \frac{np^{*}}{m}}\, dr\Big]\\
&=c_1A^{p^*} \epsilon^{\beta}\ln(\xi|\ln {\epsilon}|) \Big[\frac{1}{\theta+\beta+1} + \epsilon^{\frac{\theta+1}{p-1}-\beta} \int_{\epsilon}^{a_{\epsilon}} r^{\beta -1 - \frac{\theta+1}{p-1}}\, dr\Big].
\end{aligned}
\end{equation}
Since $a_{\epsilon}=O(\epsilon^{1/p})$ we  also can write 
\begin{equation}\label{I1-B}
\mathcal{E}_{t_{\epsilon}}(0,a_{\epsilon})\le  \left\{\begin{aligned}
    &O(\epsilon^{\beta}|\ln\epsilon|\ln(|\ln {\epsilon}|)) \; & \mbox{if}&\; \beta=\frac{\theta+1}{p-1}\\
    &O(\epsilon^{\beta}\ln(|\ln {\epsilon}|)) +O(\epsilon^{\frac{\theta+1+\beta}{p}}\ln(|\ln {\epsilon}|)) \;&\mbox{if}&\;\beta\not=\frac{\theta+1}{p-1}.
\end{aligned}\right.
\end{equation} 
\paragraph{\textbf{Sub-case}} $0<\tau<1/e$.\\
For any $r\in (b_{\epsilon}, 1)$,  from the same argument in \eqref{Ea1},  we have 
\begin{equation}\label{uppper-subcaso}
\begin{aligned}
 0\le  t_{\epsilon} u_{\epsilon}(r) \le \Bar{d}\left(\frac{\epsilon^{\frac{1}{p}}}{b_{\epsilon}}\right)^{\frac{n}{m}}\left(1+ \Big(\frac{\epsilon}{b_{\epsilon}}\Big)^{n}\right)^{-\frac{1}{m}}
 \end{aligned}
\end{equation}
where $\Bar{d}=2t_0A$, if $\epsilon>0$ is small enough.  Since $b_{\epsilon}=O(\epsilon^{\frac{1}{p}})$, as $\epsilon \to 0$ we have that  $\ln\tau \le \ln(\tau+t_{\epsilon}u_{\epsilon}))\le s_{\tau}$ on $(b_{\epsilon},1)$ for some positive constant $s_{\tau}$ depending only on $\tau$.  Hence, $|\ln(\tau+t_{\epsilon}u_{\epsilon}))|\le c_{\tau}$ on $(b_{\epsilon},1)$, for some $c_{\tau}>1$.  So, $ \ln(|\ln(\tau+t_{\epsilon}u_{\epsilon}))|)\le \ln c_{\tau}$
 on $(b_{\epsilon},1)$.  Consequently, since $e^{x}-1=e^{\vartheta x} x$ for  $x\in\mathbb{R}$ and for some  $0<\vartheta=\vartheta_{x}<1$, we can write
\begin{equation}\label{I4-B-fina}
\begin{aligned}
 \mathcal{E}_{t_{\epsilon}}(b_{\epsilon},1)&= \int_{b_{\epsilon}}^{1} r^{\theta} |u_{\epsilon}|^{p^{*}} \left(e^{r^{\beta}\ln|\ln{(\tau + t_{\epsilon} u_{\epsilon}|})} -1\right)\, dr\\
 &\le \int_{b_{\epsilon}}^{1} r^{\theta} |u_{\epsilon}|^{p^{*}} \left(e^{r^{\beta}\ln c_{\tau}} -1\right)\, dr\\
 &\le C_{\tau}\int_{b_{\epsilon}}^{1} r^{\beta+\theta} |u_{\epsilon}|^{p^{*}}\, dr.\\
 &\le  C_{\tau}A^{p^*}\epsilon^{sp^*}\int_{b_{\epsilon}}^{1} r^{\theta+\beta-\frac{np^*}{m}}\,dr.
 \end{aligned}
\end{equation}
and, recalling that $b_{\epsilon}=O(\epsilon^{\frac{1}{p}})$  
\begin{equation}\label{L2-parte 22B-final}
\epsilon^{sp^*}\int_{b_{\epsilon}}^{1} r^{\theta+\beta-\frac{np^*}{m}}\,dr=\left\{\begin{aligned}
 &O(\epsilon^{\beta}|\ln\epsilon|) &\mbox{if}&\;\; \beta=\frac{\theta+1}{p-1}\\
&O(\epsilon^{\frac{\theta+1}{p-1}})+O(\epsilon^{\frac{\theta+\beta+1}{p}})&\mbox{if}& \;\;\beta\not=\frac{\theta+1}{p-1}.
\end{aligned}\right.
\end{equation}
Hence,  from \eqref{I4-B-fina} and \eqref{L2-parte 22B-final}
\begin{equation}\label{L3-B-k}
\mathcal{E}_{t_{\epsilon}}(b_{\epsilon},1)\le \left\{\begin{aligned}
 &O(\epsilon^{\beta}|\ln\epsilon|) &\mbox{if}&\;\; \beta=\frac{\theta+1}{p-1}\\
&O(\epsilon^{\frac{\theta+1}{p-1}})+O(\epsilon^{\frac{\theta+\beta+1}{p}})&\mbox{if}& \;\;\beta\not=\frac{\theta+1}{p-1}.
\end{aligned}\right.
\end{equation} 
From \eqref{I3}, \eqref{I1-B} and  \eqref{L3-B-k} we get the result for $0<\tau<1/e$.
\paragraph{\textbf{Sub-case}} $1/e\le \tau<e$.\\ 
Here, the result follows directly from  \eqref{2-3.14} and \eqref{I1-B}.  
\paragraph{\textbf{Case~B}} $e\le \tau< \infty$.\\
We have $\ln{(\tau + t_{\epsilon}u_{\epsilon})} \geq 1$ in $(0,1)$.  As in \eqref{log-logX} we can write 
\begin{equation}\label{log-logXY}
|\ln(\tau+t_{\epsilon}u_{\epsilon})| \le \xi |\ln\epsilon|\;\; \mbox{on}\;\; (0,1).
\end{equation}
We consider the partition
\begin{equation}
0<\epsilon<\epsilon^{\frac{1}{p}}<r_0<1.
\end{equation}
Note that for $r\in (0, \epsilon^{\frac{1}{p}})$, we have $0\le r^{\beta}\ln (\xi|\ln \epsilon|)\le \epsilon^{\frac{\beta}{p}}\ln (\xi|\ln \epsilon|)\to 0$ as $\epsilon\to 0$. Thus,  we can proceed as in \eqref{L1<} to write 
 \begin{equation}\label{L1<Z}
 \begin{aligned}
 \mathcal{E}_{t_{\epsilon}}(0,\epsilon^{\frac{1}{p}}) &\le \int_{0}^{\epsilon^{\frac{1}{p}}} r^{\theta} |u_{\epsilon}|^{p^{*}} \left(e^{r^{\beta}\ln(\xi|\ln \epsilon|)} -1\right)\, dr\\
 & \le  c_1 \ln(\xi|\ln {\epsilon}|)\int_{0}^{{\epsilon}^{\frac{1}{p}}} r^{\theta+ \beta}|u_{\epsilon}|^{p^*}\, dr \\
   & =c_1 \ln(\xi|\ln {\epsilon}|)\Big[\int_{0}^{\epsilon} r^{\theta+ \beta}|u_{\epsilon}|^{p^*}\, dr + \int_{\epsilon}^{\epsilon^{\frac{1}{p}}} r^{\theta+ \beta}|u_{\epsilon}|^{p^*}\, dr\Big].
\end{aligned}
\end{equation}
We have
$$
\int_{0}^{\epsilon} r^{\theta+ \beta}|u_{\epsilon}|^{p^*}\, dr =A^{p^*}\int_{0}^{\epsilon}\frac{\epsilon^{sp^*}}{(\epsilon^n+r^n)^{\frac{p^*}{m}}} r^{\theta+ \beta}\,dr\le A^{p^*}\epsilon^{(s-\frac{n}{m})^{p^*}}\int_{0}^{\epsilon} r^{\theta+ \beta}\,dr=O(\epsilon^{\beta}).
$$
In addition,
\begin{equation}
\begin{aligned}
\int_{\epsilon}^{\epsilon^{\frac{1}{p}}} r^{\theta+ \beta}|u_{\epsilon}|^{p^*}\, dr &=A^{p^*}\int_{\epsilon}^{\epsilon^{\frac{1}{p}}}\frac{\epsilon^{sp^*}}{(\epsilon^n+r^n)^{\frac{p^*}{m}}} r^{\theta+ \beta}\,dr
\le A^{p^*}\epsilon^{sp^*}\int_{\epsilon}^{\epsilon^{\frac{1}{p}}} r^{\theta+ \beta-\frac{n}{m}p^*}\,dr\\
\end{aligned}
\end{equation}
and thus
\begin{equation}\nonumber
\int_{\epsilon}^{\epsilon^{\frac{1}{p}}} r^{\theta+ \beta}|u_{\epsilon}|^{p^*}\, dr\le \left\{\begin{aligned}
 &O(\epsilon^{\beta}|\ln\epsilon|) &\mbox{if}&\;\; \beta=\frac{\theta+1}{p-1}\\
&O(\epsilon^{\frac{\theta+\beta+1}{p}})&\mbox{if}& \;\;\beta\not=\frac{\theta+1}{p-1}.
\end{aligned}\right.
\end{equation}
From \eqref{L1<Z}, we have 
\begin{equation}\label{L2-parte 1}
\mathcal{E}_{t_{\epsilon}}(0,\epsilon^{\frac{1}{p}})\le \left\{\begin{aligned}
 &O(\epsilon^{\beta}\ln(|\ln\epsilon|))+O(\epsilon^{\beta}|\ln\epsilon|\ln(|\ln\epsilon|)) &\mbox{if}&\; \beta=\frac{\theta+1}{p-1}\\
&O(\epsilon^{\beta}\ln(|\ln\epsilon|))+O(\epsilon^{\frac{\theta+\beta+1}{p}}\ln(|\ln\epsilon|))&\mbox{if}& \;\beta\not=\frac{\theta+1}{p-1}.
\end{aligned}\right.
\end{equation}
In addition, for any $r\in (\epsilon^{1/p}, 1)$,  arguing as in \eqref{Ea1} and \eqref{2-3.14}, we can deduce that
 \begin{equation}\label{2-3.14-outro caso}
\mathcal{E}_{t_{\epsilon}}(\epsilon^{\frac{1}{p}}, 1)=\int_{\epsilon^{\frac{1}{p}}}^{1} r^{\theta}|u_{\epsilon}|^{p^{*}}   \left| \ln{(\tau + t_{\epsilon}| u_{\epsilon}|)}\right|^{r^{\beta}} \, dr  - \int_{\epsilon^{\frac{1}{p}}}^{1} r^{\theta}|u_{\epsilon}|^{p^{*}} \, dr= O(\epsilon^{\frac{\theta+1}{p}}).
\end{equation}
From \eqref{L2-parte 1} and \eqref{2-3.14-outro caso} we get \eqref{Caso B-final}.
\end{proof}
\begin{definition}
  We say  a sequence $(u_{j})$ in $X^{1,p}_{R}(\alpha_0,\alpha_1)$ is a normalized concentrating sequence at origin, NCS for short,  if 
  \begin{equation}\nonumber
  \|u_j\| =1,\; u_{j} \rightharpoonup 0\;\; \mbox{weakly in}\;\; X^{1,p}_{R}(\alpha_0,\alpha_1)\;\;\mbox{and}\;\; \lim_{j\to\infty}\int_{r_{0}}^{R} r^{\alpha_{1}} |u^{\prime}_{j}|^{p} \, dr=0,\; \forall\; r_0>0.
  \end{equation}
\end{definition}
Next, we will show that the best constant $\Sigma_{p}$ in \eqref{C3-0} is an upper bound for maximal concentrated level of the functional $J(u)=\int_{0}^{1}r^{\theta} |u|^{2^{*}} \left|\ln{(\tau + |u|)}\right|^{r^{\beta}}\, dr$, with $u\in X^{1,p}_{1}(\alpha_0,\alpha_1)$.
\begin{lemma}\label{lemma NCS} Set
    $\mathcal{M} = \left\{ (u_{j}) \subset X^{1,p}_{1}(\alpha_0,\alpha_1): \, (u_{j}) \,\,\text{is NCS}\right\}.$
Then there holds
\begin{equation}
    \sup_{ (u_{j}) \in \mathcal{M}}\Big\{\limsup_{j \to \infty} \int_{0}^{1}r^{\theta} |u_{j}|^{p^{*}} \left|\ln{(\tau + |u_{j}|)}\right|^{r^{\beta}}\, dr \Big\} \leq \Sigma_{p},
\end{equation}
for  $\tau>0$ and $\beta>0$.
\end{lemma}
\begin{proof}
Let $(u_{j}) \in \mathcal{M}$.  It is sufficient to show that given $\epsilon>0$, there exist $\rho=\rho(\epsilon)\in (0,1)$ and  $j_0\in \mathbb{N}$ such that for any $j \geq j_{0}$
\begin{itemize}
    \item [(a)] $\displaystyle\int_{0}^{\rho}r^{\theta} |u_{j}|^{p^{*}} \left|\ln{(\tau + |u_{j}|)}\right|^{r^{\beta}}\, dr \leq \Sigma_{p} + \epsilon$;
    \item[(b)] $\displaystyle\int_{\rho}^{1}r^{\theta} |u_{j}|^{p^{*}} \left|\ln{(\tau + |u_{j}|)}\right|^{r^{\beta}}\, dr \leq \epsilon$.
\end{itemize}
In order to prove $(a)$,  we  first use the Lemma~\ref{C3-L1} (c.f \eqref{e13}) to obtain
\begin{equation}\label{log-logC}
\begin{aligned}
|\ln {(\tau +|u_{j}|)}|&=\Big|\ln\tau+ \ln \Big(1 +\frac{|u_j|}{\tau}\Big)\Big|\le |\ln\tau|+ \ln \Big(1 +\frac{|u_j|}{\tau}\Big)\\
&\le |\ln\tau|+\ln\Big(1+\frac{\kappa}{\tau}r^{-\frac{\alpha_1-p+1}{p}}\Big)\\
&=|\ln r|\Big(\frac{\alpha_1-p+1}{p}+\frac{\ln(\frac{\kappa}{\tau}+r^{\frac{\alpha_1-p+1}{p}})+|\ln\tau|}{|\ln r|}\Big).
\end{aligned}
\end{equation}
Thus, by choosing  $\eta=(\alpha_1+1)/p$ and  $\rho>0$ small enough, we can write  
\begin{equation}\label{loglog}
|\ln {(\tau +|u_{j}|)}|\le \eta|\ln r|, \;\; r\in (0, \rho).
\end{equation}
By using $(e^{x}-1)/x\to 1$ as $x\to 0^+$ and that the function $r\mapsto r^{\beta}\ln|\eta\ln r|$ is increasing near $0$, for some $\delta_1>0$ small and $0<\rho<\delta_1$,  \eqref{loglog} yields
\begin{equation}\label{rho-closezero}
\begin{aligned}
 \int_{0}^{\rho} r^{\theta} |u_{j}|^{p^{*}} \left( \left|\ln {(\tau +|u_{j}|)}\right|^{r^{\beta}} - 1\right)\, dr& =\int_{0}^{\rho} r^{\theta} |u_{j}|^{p^{*}} \left(e^{r^{\beta}\ln\,\left|\ln {(\tau +|u_{j}|)}\right|} - 1\right)\, dr\\
& \le \int_{0}^{\rho} r^{\theta} |u_{j}|^{p^{*}} \left(e^{r^{\beta}\ln\,\left|\eta\ln r\right|} - 1\right)\, dr\\
& \le C\int_{0}^{\rho} r^{\theta+\beta} |u_{j}|^{p^{*}}\ln|\eta\ln r| dr\\
&\le C\rho^{\beta}\ln|\eta\ln\rho|\int_{0}^{\rho} r^{\theta} |u_{j}|^{p^{*}}\,dr\\
& \le C\rho^{\beta}\ln|\eta\ln\rho|\Sigma_p.
\end{aligned}
\end{equation}
It follows that 
\begin{equation}\label{neweq1}
\begin{aligned}
       \int_{0}^{\rho} r^{\theta} |u_{j}|^{p^{*}} \left|\ln {(\tau +|u_{j}|)}\right|^{r^{\beta}}\, dr &=      \int_{0}^{\rho} r^{\theta} |u_{j}|^{p^{*}}\, dr+ \int_{0}^{\rho} r^{\theta} |u_{j}|^{p^{*}} \left( \left|\ln {(\tau +|u_{j}|)}\right|^{r^{\beta}} - 1\right)\, dr\\
       &\leq  \Sigma_{p} +  C \rho^{\beta} \ln{|\eta\ln{\rho}|} \Sigma_{p}\\
       &\leq \Sigma_{p}+ \epsilon,
       \end{aligned}
\end{equation}
for $\rho = \rho(\epsilon)\in (0,\delta_{1})$ sufficiently small.
To get $(b)$, for any $r \in (\rho(\epsilon), 1)$, we obtain
\begin{equation}\label{3.24-}
\begin{aligned}
    |u_{j}(r)| &\leq \int_{r}^{1} |u'_{j}(s)|\, ds = \int_{r}^{1} s^{\frac{\alpha_{1}}{p}} |u'_{j}(s)| s^{-\frac{\alpha_{1}}{p}}\, ds \\
    &\leq \left(\int_{\rho}^{1}s^{\alpha_{1}}|u'_{j}|^{p}\, ds\right)^{\frac{1}{p}} \left(\int_{r}^{1}s^{\frac{-\alpha_{1}}{p-1}}\, ds\right)^{\frac{p-1}{p}}\\
    &\leq \kappa_{j} r^{-\frac{\alpha_{1}-p+1}{p}},
    \end{aligned}
\end{equation}
where
$\kappa_{j} := C \left(\int_{\rho}^{1}s^{\alpha_{1}}|u'_{j}|^{p}\, ds\right)^{\frac{1}{p}},$
for some $C= C(\alpha_{1}, p)$. Since  $(u_{j})$ is NCS, we have $\kappa_{j} \to 0$, as $j \to \infty$. From \eqref{3.24-}, if $\tau\in (0,1)$ then 
\begin{equation}\nonumber
\begin{aligned}
 |\ln {(\tau +|u_{j}|)}|& \le |\ln\tau|+\ln\Big(1+\frac{\kappa_{j}}{\tau}r^{-\frac{\alpha_1-p+1}{p}}\Big)\\
& \le |\ln{\tau}| + \ln\left(1 +\frac{\kappa_{j}}{\tau}\rho^{-\frac{\alpha_{1}-p+1}{p}}\right)\leq C_1, \quad \text{for}\,\, r \in (\rho, 1)
\end{aligned}
\end{equation}
and, if $\tau\ge 1$ then
$$|\ln {(\tau +|u_{j}|)}| \leq  \ln{\left(\tau + \kappa_{j}\rho^{-\frac{\alpha_{1}-p+1}{p}}\right)}\leq C_2, \quad \text{for}\,\, r \in (\rho, 1)$$
for $j$ large enough, where $C_1$ and $C_2$ depend only on $\rho$ and $\tau$. For $C=\max\{1,C_1, C_2\}$ we obtain 
\begin{equation}\nonumber
\begin{aligned}
    \int_{\rho}^{1} r^{\theta} |u_{j}|^{p^{*}} \left|\ln{(\tau + |u_{j}|)}\right|^{r^{\beta}}\, dr &\leq C \int_{\rho}^{1} r^{\theta}(\kappa_{j} r^{-\frac{\alpha_{1}-p+1}{p}})^{p^{*}}\, dr \\
    &= C\kappa_{j}^{p^{*}} \int_{\rho}^{1} \frac{1}{r}\, dr < \epsilon,
    \end{aligned}
\end{equation}
for $j \in \mathbb{N}$ large enough.
 \end{proof}
  \begin{proof}[Proof of the Theorem \ref{T2}] 
By choosing  $\widehat{A} = \mathcal{S}^{-\frac{\theta+1}{(\theta-\alpha_{1}+p)p}}$ in \eqref{definition u},  from \eqref{S-E},  \eqref{3.6} and \eqref{3.7}, we have
\begin{equation}\label{3.24}
    \|u_{\epsilon}\|^{p} = 1 + O(\epsilon^{sp}) \quad \text{and}\quad    \|u_{\epsilon}\|^{p^*}_{L_{\theta}^{p^{*}}} = \Sigma_{p} + O(\epsilon^{sp^{*}}).
\end{equation}
Firstly, the Lemma \ref{lema20} with $t_{\epsilon}=1/(1+O(\epsilon^{sp}))$ and \eqref{3.24}   imply
\begin{equation}\label{3.25}
\begin{aligned}
\int_{0}^{1}r^{\theta} |u_{\epsilon}|^{p^{*}} \Big| & \ln{\Big(\tau +\frac{|u_{\epsilon}|}{1 + O(\epsilon^{sp})}\Big)}\Big|^{r^{\beta}}\, dr=\mathcal{E}_{t_{\epsilon}}(0,1)+ \|u_{\epsilon}\|^{p^*}_{L_{\theta}^{p^{*}}}\\
     & \ge      \left\{
\begin{aligned}
       &\Sigma_p+ O(\epsilon^{\frac{\theta+1}{p-1}}) +  C \epsilon^{\beta} \ln {|\ln {\epsilon}|} + O(\epsilon^{\frac{\theta+1}{p}}), &\;\; \mbox{if} \;\;&  0<\tau < e \\
     &  \Sigma_p+ O(\epsilon^{\frac{\theta+1}{p-1}})  +  C \epsilon^{\beta} \ln {|\ln {\epsilon}|}, &\;\; \mbox{if}\; \;&   e \leq \tau < \infty
\end{aligned}
\right. \\
\end{aligned}
\end{equation}
for some $C>0$. Hence, \eqref{3.24} and \eqref{3.25}  yield
\begin{equation}\label{F>quase}
\begin{aligned}
    \mathcal{F}_{ \tau, \beta, \theta} 
    &\ge \int_{0}^{1}r^{\theta} \Big(\frac{|u_{\epsilon}|}{\|u_{\epsilon}\|}\Big)^{p^{*}} \Big| \ln{\Big(\tau+\frac{|u_{\epsilon}|}{\|u_{\epsilon}\|}\Big)}\Big|^{r^{\beta}}\, dr \\
   &= (1+O(\epsilon^{sp}))\int_{0}^{1}r^{\theta} |u_{\epsilon}|^{p^{*}} \Big| \ln{\Big(\tau+\frac{|u_{\epsilon}|}{1 + O(\epsilon^{sp})}\Big)}\Big|^{r^{\beta}}\, dr \\
&\ge      \left\{
\begin{aligned}
       &\Sigma_p+ O(\epsilon^{\frac{\theta+1}{p-1}}) +  C \epsilon^{\beta} \ln {|\ln {\epsilon}|} + O(\epsilon^{\frac{\theta+1}{p}})+O(\epsilon^{sp}), &\;\; \mbox{if} \;\;&  0<\tau < e \\
     &  \Sigma_p+ O(\epsilon^{\frac{\theta+1}{p-1}})  +  C \epsilon^{\beta} \ln {|\ln {\epsilon}|}+O(\epsilon^{sp}), &\;\; \mbox{if}\; \;&   e \leq \tau < \infty
\end{aligned}
\right.
\end{aligned}
\end{equation}
if $\epsilon>0$ is small enough, for all $\tau>0$ and $\beta>0$. 

\paragraph{$(i)$}  Letting $\epsilon \to 0$ in \eqref{F>quase}, we get $\mathcal{F}_{ \tau, \beta, \theta} \geq \Sigma_{p}$ for all $\tau>0$ and $\beta>0$. 

\paragraph{$(ii)$}  From \eqref{F>quase}, we can write 
 \begin{equation}\label{IIF>quase}
\begin{aligned}
    \mathcal{F}_{ \tau, \beta, \theta} &\ge      \left\{
\begin{aligned}
       &\Sigma_p+ \epsilon^{\beta} \ln {|\ln {\epsilon}|} \Big[C+O\Big(\frac{\epsilon^{\frac{\theta+1}{p-1}-\beta}}{ \ln {|\ln {\epsilon}|} }\Big) + O\Big(\frac{\epsilon^{\frac{\theta+1}{p}-\beta}}{ \ln {|\ln {\epsilon}|} }\Big)+O\Big(\frac{\epsilon^{sp-\beta}}{\ln|\ln\epsilon|}\Big)\Big], &\; \mbox{if} \;&  0<\tau < e \\
     &  \Sigma_p+\epsilon^{\beta} \ln {|\ln {\epsilon}|}\Big[C+O\Big(\frac{\epsilon^{\frac{\theta+1}{p-1}-\beta}}{ \ln {|\ln {\epsilon}|} }\Big)+O\Big(\frac{\epsilon^{sp-\beta}}{\ln|\ln\epsilon|}\Big)\Big],  &\; \mbox{if}\; &   e \leq \tau <\infty.
\end{aligned}
\right.
\end{aligned}
\end{equation}
Taking into account the assumption  $0< \beta< \min \{(\theta+1)/p, sp\}$, \eqref{IIF>quase} yields $(ii)$.

\paragraph{$(iii)$}
In view of $(i)$, in order to get  $(iii)$,  it is enough to show that
\begin{equation}
\lim_{\beta \to \infty} \sup{\mathcal{F}_{\tau, \beta, \theta}}\le \Sigma_{p}.
\end{equation}
By contradiction,  suppose that there exists a sequence $(\beta_{j})$ such that
\begin{equation}\label{3.26}
    \lim_{j \to \infty} \beta_{j} = \infty \quad \text{and}\quad \Sigma_{p} < \lim_{j \to \infty} \mathcal{F}_{\tau, \beta_{j}, \theta}.
\end{equation}
Of course, we can assume  $\beta_{j}>1$ for any $j\in\mathbb{N}$. For each $j$, we can take $u_{j} \in X^{1,p}_{1}(\alpha_0,\alpha_1)$ such that
\begin{equation}\label{3.27-}
    \|u_{j}\|=1 \quad \text{and} \quad \mathcal{F}_{\tau, \beta_{j}, \theta} - \frac{1}{j} \leq \int_{0}^{1} r^{\theta} |u_{j}|^{p^{*}} |\ln{(\tau + |u_{j}|)|^{r^{\beta_{j}}}}\, dr.
\end{equation}
Passing to a subsequence if necessary, the compact embedding \eqref{eq10} implies  that there exists $u_{0} \in X^{1,p}_{1}(\alpha_0,\alpha_1)$ such that
\begin{equation}\label{up-compactness}
    u_{j} \rightharpoonup u_{0} \,\, \text{weakly in}\,\, X^{1, p}_{1}(\alpha_0,\alpha_1),\,\, u_{j} \to u_{0}\,\, \text{in}\,\, L_{\theta}^{q} \,\, \text{and}\,\, 
    u_{j} \to u_{0} \,\, \text{a. e. in}\,\, (0,1),
\end{equation}
for any $1<q <p^{*}$. 
Let us denote by $X_{1}^{1,p}([\rho,1))$ the space $X_{1}^{1,p}(\alpha_0,\alpha_1)$ on the interval $(\rho, 1]$ instead of $(0,1]$.  According to \cite[pag. 3356]{DCDS2019}, we have  the compact embedding 
$$X^{1, p}_{1}([\rho, 1))\hookrightarrow  L^{q}_{\theta}([\rho, 1))$$
for any $q \geq p$. From \eqref{up-compactness} and  Lemma~\ref{C3-L1}, we can apply the Lebesgue dominated convergence theorem to see that
\begin{eqnarray*}
    \lim_{j \to \infty} \int_{\rho}^{1} r^{\theta} |u_{j}|^{p^{*}} |\ln{(\tau+ |u_{j}|)}|^{r^{\beta_{j}}}\, dr  =  \int_{\rho}^{1} r^{\theta} |u_{0}|^{p^{*}}\, dr = \lim_{j \to \infty} \int_{\rho}^{1} r^{\theta} |u_{j}|^{p^{*}}\, dr,     
\end{eqnarray*}
for any fixed $\rho \in (0,1)$. Then, for $\tau > 0$ holds
\begin{equation}\label{3.27}
        \int_{\rho}^{1} r^{\theta} |u_{j}|^{p^{*}} |\ln{(\tau+ |u_{j}|)}|^{r^{\beta_{j}}}\, dr  = \int_{\rho}^{1} r^{\theta} |u_{j}|^{p^{*}}\, dr + o_{j}(1).
\end{equation}
Arguing as in \eqref{log-logC} we can write
\begin{equation}\nonumber
\begin{aligned}
|\ln {(\tau +|u_{j}|)}|&\le|\ln r|\Big(\frac{\alpha_1-p+1}{p}+\frac{\ln(\frac{\kappa}{\tau}+r^{\frac{\alpha_1-p+1}{p}})+|\ln\tau|}{|\ln r|}\Big).
\end{aligned}
\end{equation}
Hence, for $0<\rho<\delta_0$  small enough and by choosing $C_{0}=\max\{\frac{\alpha_1+1}{p}, 1\}$ we have 
\begin{equation}\nonumber
\begin{aligned}
|\ln {(\tau +|u_{j}|)}|&\le C_{0}|\ln r|,\;\; \mbox{for all}\;\; r\in (0,\rho).
\end{aligned}
\end{equation}
Since $\beta_j>1$ and $C_0\ge 1$ we  also can write 
\begin{equation}\nonumber
\begin{aligned}
|\ln {(\tau +|u_{j}|)}|^{r^{\beta_j}}&\le \big(C_{0}|\ln r|\big)^{r^{\beta_j}}\le \big(C_{0}|\ln r|\big)^{r}, \;\; \mbox{for all}\;\; r\in (0,\rho).
\end{aligned}
\end{equation}
Since $\big(C_{0}|\ln r|\big)^{r}\to 1$ as $r\to 0$, for any $\epsilon>0$ we have 
\begin{equation*}
    |\ln {(\tau +|u_{j}|)}|^{r^{\beta_{j}}} -1 \le  \big(C_{0}|\ln r|\big)^{r} -1< \epsilon, 
\end{equation*}
for $r$ near $0$. Then, choosing $\rho>0$ small enough we can deduce that
\begin{equation}\label{3.32-}
    \int_{0}^{\rho} r^{\theta} |u_{j}|^{p^{*}} |\ln{(\tau+ |u_{j}|)}|^{r^{\beta_{j}}} \, dr \leq (1+ \epsilon) \int_{0}^{\rho} r^{\theta} |u_{j}|^{p^{*}}\, dr.
\end{equation}
Combining \eqref{3.27-}, \eqref{3.27} and \eqref{3.32-} we have
\begin{eqnarray*}
    \mathcal{F}_{\tau, \beta_{j}, \theta}  - \dfrac{1}{j} \leq (1+\epsilon)\Sigma_{p}+ o_{j}(1).
\end{eqnarray*}
Letting $j \to \infty$ and then $\epsilon \to 0$ we get a contradiction.
\end{proof}

\section{Attainability: Proof of Theorem \ref{C3-T3}}\label{sec4}
This section is dedicated to proving our attainability result. The main idea is based on the two-step argument due to Carleson-Chang, which first appeared in the seminal work \cite{CC}: Precisely,  under the conditions $\tau\ge 1$ and  $0<\beta< \min\{(\theta+1)/p, (\alpha_{1}-p+1)/(p-1)\}$, we will prove the following:
\begin{enumerate}
\item [I)] $ \mathcal{F}_{\tau, \beta, \theta}> \Sigma_p$,
\item [II)] If $ \mathcal{F}_{\tau, \beta, \theta}$ is not attained, then $ \mathcal{F}_{\tau, \beta, \theta}\le \Sigma_p$.
\end{enumerate}
It is clear that the contradiction derived from I) and II) ensures the existence of extremal function for $\mathcal{F}_{\tau, \beta, \theta}$.  Note that Theorem~\ref{T2}-$(ii)$ ensures that I) holds.   In view of the Lemma~\ref{lemma NCS}, to conclude II) we only need to show the following result. 
\begin{proposition} Let $\tau\ge 1$ and $\beta>0$.  If $ \mathcal{F}_{\tau, \beta, \theta}$ is not attained, then any of its maximizing sequence is necessarily  NCS.
\end{proposition}
\begin{proof}
Let  $(u_{j})\subset X^{1,p}_{1}(\alpha_0,\alpha_1)$ be  a maximizing sequence for  $\mathcal{F}_{\tau, \beta, \theta}$, i.e. it satisfies
\begin{equation}\label{4.1}
    \|u_j\| = 1\quad \text{and}\quad \mathcal{F}_{\tau, \beta, \theta}= \lim_{j \to \infty} \int_{0}^{1} r^{\theta} |u_{j}|^{p^{*}} |\ln{(\tau + |u_{j}|)}|^{r^{\beta}}\, dr.
\end{equation}
 Up to a subsequence if necessary, we can take $u \in X^{1,p}_{1}(\alpha_0,\alpha_1)$ such that $u_{j} \rightharpoonup u$ weakly in $X_{1}^{1,p}(\alpha_0,\alpha_1)$, $u_{j} \to u$ a.e. in $(0,1)$ and   $\|u\| \leq \liminf\|u_{j}\| = 1$.
 
 We will divide  our proof into two steps:
 
\paragraph{\textbf{Step~1:}} We prove that $u\equiv0$. Suppose  that $u\not\equiv0$. Set $v_{j} = u_{j}-u$. By  using Brezis-Lieb type argument in \cite{BL},
we obtain
\begin{eqnarray}\label{4.2}
        \int_{0}^{1} r^{\theta} |u_{j}|^{p^{*}} \left| \ln{(\tau + |u_{j}|)}\right|^{r^{\beta}} \, dr&=&     \int_{0}^{1} r^{\theta} |v_{j} |^{p^{*}} \left| \ln{(\tau + |v_{j}|)}\right|^{r^{\beta}} \, dr \nonumber\\
    &+&     \int_{0}^{1} r^{\theta} |u|^{p^{*}} \left| \ln{(\tau + |u|)}\right|^{r^{\beta}}\, dr + o_{j}(1),    
\end{eqnarray}
and 
\begin{eqnarray}\label{4.3}
        1= \int_{0}^{1} r^{\alpha_{1}} |u'_{j}|^{p} \, dr =     \int_{0}^{1} r^{\alpha_{1}} |v'_{j}|^{p} \, dr
    +     \int_{0}^{1} r^{\alpha_{1}} |u'|^{p} \, dr + o_{j}(1),    
\end{eqnarray}
where $o_j(1) \to 0$ as $j \to \infty$.  If $\|u\| = 1$ in \eqref{4.3}, we obtain $u_{j} \to u$ strongly in $X^{1,p}_{1}(\alpha_0,\alpha_1)$ or $v_j\to 0$ strongly in $X^{1,p}_{1}(\alpha_0,\alpha_1)$. Thus, by combining the continuous embedding in Corollary~\ref{C3-C1} with \eqref{4.2} we have that $u$ is an extremal function for  $\mathcal{F}_{\tau, \beta, \theta}$, which contradicts our assumption. Therefore, we may assume that $0<\|u\| < 1$. In this case, from \eqref{4.1}, \eqref{4.2} and \eqref{4.3}  we can write
\begin{equation}
\begin{aligned}
    \mathcal{F}_{\tau, \beta, \theta} 
    &= \|v_{j}\|^{p^{*}} \int_{0}^{1} r^{\theta} \Big(\dfrac{|v_{j}|}{\|v_{j}\|}\Big)^{p^{*}} \left| \ln{(\tau + |v_{j}|)}\right|^{r^{\beta}}\, dr \nonumber\\
    &+    \|u\|^{p^*} \int_{0}^{1} r^{\theta} \Big(\dfrac{|u|}{\|u\|}\Big)^{p^{*}} \left| \ln{(\tau + |u|)}\right|^{r^{\beta}}\, dr + o_{j}(1) \nonumber\\   
    &\le \|v_{j}\|^{p^{*}} \int_{0}^{1} r^{\theta} \Big(\dfrac{|v_{j}|}{\|v_{j}\|}\Big)^{p^{*}} \Big| \ln\Big(\tau + \Big|\frac{v_{j}}{\|v_j\|}\Big|\Big)\Big|^{r^{\beta}}\, dr \nonumber\\
   & +  \|u\|^{p^{*}} \int_{0}^{1} r^{\theta} \Big(\dfrac{|u|}{\|u\|}\Big)^{p^{*}} \Big| \ln\Big(\tau + \Big|\frac{u}{\|u\|}\Big|\Big)\Big|^{r^{\beta}}\, dr +o_j(1)\nonumber\\
  & \le \mathcal{F}_{\tau, \beta, \theta} \left(\|v_j\|^{p^{*}} +\|u\|^{p^{*}}\right) + o_{j}(1)\nonumber\\
&= \mathcal{F}_{\tau, \beta, \theta}\left(\left(1-\|u\|^p+ o_{j}(1) \right)^{\frac{p^{*}}{p}}+\left(\|u\|^p\right)^{\frac{p^{*}}{p}}\right) + o_{j}(1)\nonumber\\
&<  \mathcal{F}_{\tau, \beta, \theta},
\end{aligned}
\end{equation}
which is an contradiction, where  we used that $\left(1-t\right)^{\frac{p^{*}}{p}}+t^{\frac{p^{*}}{p}}<1$ for all $t \in (0, 1)$. Then $u\equiv 0$, as desired. 

\paragraph{\textbf{Step~2:}} For each $r_{0}\in (0, 1)$ holds
\begin{equation}\label{4.6}
    \int_{r_{0}}^{1} r^{\alpha_{1}} |u_{j}'|^{p} \to 0, \quad \text{as} \,\, j \to 0. 
\end{equation}
From Lemma \ref{C3-L1}  we have 
\begin{equation}\label{4.6-}
 C_{0} = C_{0}(\alpha_{1},\beta, \tau, r_{0})= \sup_{r \in [r_{0}, 1]}|\ln{(\tau + |u_{j}|)}|^{r^{\beta}} <\infty.
\end{equation}
Hence, the compact embedding 
\begin{equation}\label{R0-compact}
X^{1, p}_{1}([r_{0}, 1])\hookrightarrow  L^{q}_{\theta}([r_{0}, 1]),\;\;  q\ge p
\end{equation} yields
\begin{eqnarray}\label{4.7-}
    \int_{r_{0}}^{1} r^{\theta} |u_{j}|^{p^{*}} |\ln{(\tau + |u_{j}|)}|^{r^{\beta}}\, dr 
    &\leq& C_{0} \int_{r_{0}}^{1} r^{\theta} |u_{j}|^{p^{*}}\, dr \to 0, \quad \text{as}\,\,j \to \infty.
\end{eqnarray}
Since $(u_{j})$ is a maximizing sequence, by the Ekeland's variational principle \cite[Theorem~3.1]{Ekeland} there exists a multiplier $\lambda_j$ such that 
\begin{equation}\label{4.7}
\begin{aligned}
    \lambda_{j} \int_{0}^{1} r^{\alpha_{1}} |u_{j}'|^{p-2} u'_{j}\varphi'\, dr &= p^{*}\int_{0}^{1} r^{\theta} |u_{j}|^{p^{*}-2} \left(\ln{(\tau +|u_{j}|)}\right)^{r^{\beta}}u_{j}\varphi\, dr \\
    &+  \int_{0}^{1} \dfrac{r^{\theta+\beta} |u_{j}|^{p^{*}-1} u_{j} \varphi}{ (\tau +|u_{j}|)\left(\ln{(\tau +|u_{j}|)}\right)^{1- r^{\beta}}} \, dr + \langle o_{j}(1), \varphi \rangle.
    \end{aligned}
\end{equation}
By choosing $\varphi = u_{j}$ in \eqref{4.7} we obtain
\begin{eqnarray*}
    \lambda_{j} \int_{0}^{1} r^{\alpha_{1}} |u_{j}'|^{p} \, dr 
 &\geq& p^{*}\int_{0}^{1} r^{\theta} |u_{j}|^{p^{*}} \left(\ln{(\tau +|u_{j}|)}\right)^{r^{\beta}}\, dr + \langle o_{j}(1), u_{j} \rangle. \nonumber\\
\end{eqnarray*}
Letting $j \to \infty$, it  follows that $\liminf \lambda_{j} \geq p^{*} \mathcal{F}_{\tau, \beta, \theta}$. Let $\Bar{\eta}:[0,1]\to[0,1]$ be a smooth function such that  $\Bar{\eta} \equiv 0$ in $[0, r_{0}/2]$ and $\Bar{\eta} \equiv 1$ in $[r_{0}, 1]$. From \eqref{R0-compact}, we have 
\begin{equation}\label{4.8}
\begin{aligned}
    \int_{0}^{1} \dfrac{r^{\theta+\beta} |u_{j}|^{p^{*}-1} u_{j} (\Bar{\eta} u_{j})}{ (\tau +|u_{j}|)\left(\ln{(\tau +|u_{j}|)}\right)^{1- r^{\beta}}} \, dr &=     \int_{r_{0}/2}^{1} \dfrac{r^{\theta+\beta} |u_{j}|^{p^{*}+1}\Bar{\eta} }{ (\tau +|u_{j}|)\left(\ln{(\tau +|u_{j}|)}\right)^{1- r^{\beta}}} \, dr \\
    &\leq   \int_{r_{0}/2}^{1} \dfrac{r^{\theta} |u_{j}|^{p^{*}} }{\left(\ln{(\tau +|u_{j}|)}\right)^{1- r^{\beta}}} \, dr\\
    &\leq   C\int_{r_{0}/2}^{1} r^{\theta} |u_{j}|^{p^{*}}  \, dr \to 0, 
    \end{aligned}
    \end{equation}
where we  used $0\leq |u_{j}|< |u_{j}| + \tau$, $\left(\ln{(\tau +|u_{j}|)}\right)^{r^{\beta}} \leq C \ln{(\tau +|u_{j}|)}$ for $\tau \in [1, \infty)$.
Thus, by taking $\varphi = \Bar{\eta} u_{j}$ in \eqref{4.7} and combining
 \eqref{4.7-} and \eqref{4.8}  we get
\begin{equation}\label{upp-dow}
\begin{aligned}
    o_{j}(1) &= \int_{r_{0}/2}^{1} r^{\alpha_{1}} |u_{j}'|^{p-2} u'_{j} (\Bar{\eta} u_{j})'\, dr\\
    &= \int_{r_{0}/2}^{1} r^{\alpha_{1}} |u_{j}'|^{p} \Bar{\eta}\, dr+ \int_{r_{0}/2}^{1} r^{\alpha_{1}} |u_{j}'|^{p-2} u'_{j}  u_{j} \Bar{\eta}'\, dr\\
    &\geq \int_{r_{0}}^{1} r^{\alpha_{1}} |u_{j}'|^{p}\, dr - \|\Bar{\eta}'\|_{\infty} \|u_{j}'\|_{L^{p}_{\alpha_{1}}}^{p-1}   \left(\int_{r_{0}}^{1}r^{\alpha_1}|u_{j}|^{p} \, dr\right)^{\frac{1}{p}}\\
   &= \int_{r_{0}}^{1} r^{\alpha_{1}} |u_{j}'|^{p}\, dr + o_{j}(1)
   \end{aligned}
\end{equation}
where we used $\int_{r_{0}}^{1}r^{\alpha_1}|u_{j}|^{p} \, dr\le C\int_{r_0}^{1}r^{\theta}|u_{j}|^{p} \, dr\to 0$ as $j\to\infty$. Finally, \eqref{upp-dow} ensures \eqref{4.6} holds.
\end{proof}
\section{Application to a class of quasilinear elliptic equations}\label{sec24}
 In this section we prove the existence of a nontrivial weak solution to problem \eqref{m.5} stated in  Theorem \ref{C3-T4}. To achieve our result, we apply variational arguments to the functional  $I:X^{1,p}_{1}(\alpha_0,\alpha_1) \to \mathbb{R}$ defined by
\begin{equation}\label{funcional - I(u)}
    I(u) = \frac{1}{p} \int_{0}^{1} r^{\alpha_{1}} |u'|^{p} \, dr - \dfrac{1}{p^{*}}\int_{0}^{1}  r^{\theta} |u|^{p^{*}} (\ln{(\tau + |u|)})^{r^{\beta}}\, dr + \int_{0}^{1} r^{\theta} G(r, u)  \,dr,
\end{equation}
where
$$G(r, u)= \int_{0}^{u} g(r, s) \, ds, \quad g(r, s)= \dfrac{r^{\beta} |s|^{p^{*}-1}s}{p^{*}(\tau+ |s|)(\ln{(\tau+ |s|)})^{1-r^{\beta}}}.$$
Note that $g(0, u) =0$ and,  for any small $\epsilon>0$, there exists a large constant $C_{\epsilon}>0$ such that
\begin{equation}\label{5.3}
    |G(r, u)| \leq r^{\beta} (\epsilon |u|^{p^{*}} + C_{\epsilon} |u|^{p}) \quad \text{and} \quad     |G(r, u)| \leq r^{\beta} (\epsilon |u|^{p} + C_{\epsilon} |u|^{p^{*}}).
\end{equation}
From Theorem \ref{C3-T1},  we can deduce that the functional $I$ is well-defined and of class $C^{1}$ on $X^{1,p}_{1}(\alpha_0,\alpha_1)$. In fact, for all $\varphi \in X^{1,p}_{1}(\alpha_0,\alpha_1)$ we have
\begin{equation}\label{funcional derivada - I'(u)}
     \langle I'(u), \varphi \rangle = \int_{0}^{1}  r^{\alpha_{1}} |u'|^{p-2}  u'\varphi'\, dr - \int_{0}^{1} r^{\theta} |u|^{p^{*}-1} \left(\ln{(\tau + |u|)}\right)^{r^{\beta}} \varphi \, dr.
\end{equation}
So, the critical points of the functional $I$ are weak solutions of \eqref{m.5}. As in \cite{MR3514752},  we will apply a version of the mountain pass theorem without the Palais-Smale condition due to Brezis and Nirenberg, see \cite[Theorem 2.2]{MR0709644} to obtain critical points of $I$. We will proceed in the following steps: 
\begin{itemize}
     \item[(A)] [Lemma~\ref{geometria PM}] The functional $I$ has the mountain pass geometry.
     
    \item [(B)] [Lemma \ref{lema B}] The mountain pass level $c_{MP}$ satisfies
    $$c_{MP} < \Big(\frac{1}{p}-\frac{1}{p^*}\Big)\mathcal{S}^{\frac{\theta+1}{\theta-\alpha_1+p}}.$$
    
     \item [(C)][Lemma \ref{lema A}] There is a loss of compactness for the functional $I$ at the level 
$$ \Big(\frac{1}{p}-\frac{1}{p^*}\Big)\mathcal{S}^{\frac{\theta+1}{\theta-\alpha_1+p}}.$$    

    \item [(D)] [Lemma \ref{seq limitada}] There is a nontrivial weak solution at the level $c_{MP}>0$. 
\end{itemize}

\begin{lemma}\label{geometria PM}
The functional $I$ has the  mountain pass geometry. 
 \begin{itemize}
     \item[$(a)$] $I(0)=0$.
    \item [$(b)$] For each $u \in X^{1,p}_{1}(\alpha_0,\alpha_1)\setminus \{0\}$ with $u \geq 0$, we have
    $I(tu) \to - \infty,$ if $ t \to +\infty.$
    \item [$(c)$] There are $\delta, \rho>0$ such that $I(u) \geq \delta$, if $\|u\| = \rho.$
 \end{itemize}
\end{lemma}
\begin{proof}
Obviously $I(0)=0$.  Let  $u \in X^{1,p}_{1}(\alpha_0,\alpha_1)$ with $u \geq 0$.  If $t \ge 1 $, then $\ln(\tau+tu)\ge \ln(\tau+u)$ on $(0,1)$. Then, for all $t\ge1$,  \eqref{5.3} yields
\begin{equation}
\begin{aligned}
    I(t u)   &\le  \dfrac{t^{p}}{p} \int_{0}^{1} r^{\alpha_{1}} |u'|^{p} \, dr - \dfrac{t^{p^{*}}}{p^{*}}\int_{0}^{1}  r^{\theta} |u|^{p^{*}} (\ln{(\tau + |u|)})^{r^{\beta}}\, dr \\
    &+\epsilon t^{p^*}\int_{0}^{1} r^{\theta+\beta}|u|^{p^*} \, dr+C_\epsilon  t^{p}\int_{0}^{1} r^{\theta+\beta}|u|^{p} \, dr\\
    &\le  \dfrac{t^{p}}{p} \|u\|^p - t^{p^{*}}\Big[ \dfrac{1}{p^{*}}\int_{0}^{1}  r^{\theta} |u|^{p^{*}} (\ln(\tau + |u|))^{r^{\beta}}\,dr-\epsilon\|u\|^p_{L^p_{\theta}}\Big]+C_\epsilon  t^{p}\|u\|^{p}_{L^p_{\theta}}.
    \end{aligned}
\end{equation}
Since $p^{*}>p$, by choosing $\epsilon>0$ small enough so that the term in the bracket is positive, and letting  $t \to \infty$ we get $(b)$. To prove $(c)$ we note that, for any $u \in X^{1,p}_{1}(\alpha_0,\alpha_1)$ with $0<\|u\|<1$, Theorem \ref{C3-T1} implies 
\begin{eqnarray*}
    \frac{1}{\|u\|^{p^*}}\int_{0}^{1} r^{\theta} |u|^{p^{*}} |\ln{(\tau+|u|)}|^{r^{\beta}}\, dr \leq \int_{0}^{1} r^{\theta} \Big| \dfrac{u}{\|u\|}\Big|^{p^*} \left|\ln{\left(\tau+    \Big| \dfrac{u}{\|u\|}\Big|\right)}\right|^{r^{\beta}}\, dr\leq \mathcal{F}_{\tau, \beta, \theta}
\end{eqnarray*}
which means that 
\begin{equation}\label{5.4}
       \int_{0}^{1} r^{\theta} |u|^{p^{*}} |\ln{(\tau+|u|)}|^{r^{\beta}}\, dr \leq \mathcal{F}_{\tau, \beta, \theta} \|u\|^{p^{*}}, \quad \text{if}\,\,\|u\|<1.
\end{equation}
From the continuous embedding \eqref{eq10} and the estimates  \eqref{5.3} and \eqref{5.4} we have
\begin{equation}
\begin{aligned}
    I(u)  &\geq  \frac{1}{p} \|u\|^{p}  - \dfrac{\mathcal{F}_{\tau, \beta, \theta} }{p^{*}} \|u\|^{p^{*}} - \int_{0}^{1}r^{\theta+ \beta} (\epsilon|u|^{p}+ C_{\epsilon}|u|^{p^{*}})\, dr \\
    &\geq  \frac{1}{p} \|u\|^{p}  - \dfrac{\mathcal{F}_{\tau, \beta, \theta} }{p^{*}} \|u\|^{p^{*}} - \int_{0}^{1}r^{\theta} (\epsilon|u|^{p}+ C_{\epsilon}|u|^{p^{*}})\, dr \\
   &=  \frac{1}{p} \|u\|^{p}  - \dfrac{\mathcal{F}_{\tau, \beta, \theta} }{p^{*}} \|u\|^{p^{*}} - \epsilon\|u\|^{p}_{L^{p}_{\theta}} -  C_{\epsilon} \|u\|^{p^{*}}_{L^{p^{*}}_{\theta}}\\
   &\ge  \frac{1}{p} \|u\|^{p}  - \dfrac{\mathcal{F}_{\tau, \beta, \theta} }{p^{*}} \|u\|^{p^{*}} - \epsilon C_1\|u\|^{p} -  C_{\epsilon}C_2 \|u\|^{p^{*}}\\
     &= \Big(\frac{1}{p} - \epsilon C_{1} \Big)\|u\|^{p}  - \dfrac{\mathcal{F}_{\tau, \beta, \theta} + p^{*}C_{2}C_{\epsilon}}{p^{*}} \|u\|^{p^{*}}.
     \end{aligned}
\end{equation}
Since $p^*>p$,  we can choose $\rho>0$ small enough and $\delta>0$ satisfying $I(u)\geq \delta$ for $\|u\|= \rho$. This prove that $(c)$ holds.
\end{proof}
 Let us take $u_{\epsilon}$ as in \eqref{definition u} with $\widehat{A}=1$. 
In view of Lemma \ref{geometria PM}, we can take the mountain pass level
\begin{equation}\label{Nivel MP} 
     c_{MP}= \inf_{\gamma \in \Gamma} \max_{u \in \gamma} I(u)
\end{equation}
where 
$$\Gamma = \{\gamma : [0,T] \to  X^{1,p}_{1}(\alpha_0,\alpha_1)\;: \;\gamma \,\,\text{is continuous}, \,\, \gamma(0)=0\,\mbox{and}\;\; \gamma(T)=Tu_{\epsilon}\}$$
with  $T>0$ large enough, so that $I(Tu_{\epsilon})<0$. By Lemma \ref{geometria PM}  follows that $\Gamma \neq \emptyset$ and $c_{MP} \geq  \delta > 0$.
\begin{lemma}\label{lema B} The mountain pass level $c_{MP}$ satisfies
    $$0<c_{MP} < \Big(\frac{1}{p}-\frac{1}{p^*}\Big)\mathcal{S}^{\frac{\theta+1}{\theta-\alpha_1+p}}.$$
\end{lemma}
\begin{proof}
 For each $\epsilon>0$, consider the path $\gamma_{\epsilon} \in \Gamma$ given by $\gamma_{\epsilon} (t) = t u_{\epsilon}$, with $t\in [0,T]$.  Then,  from the definition of \eqref{Nivel MP},  there exists $t_{\epsilon}>0$ such that 
\begin{equation}\label{5.6}
    (I \circ \gamma_{\epsilon})(t_{\epsilon}) = \displaystyle\max_{t \in [0, T]} (I \circ \gamma_{\epsilon}) (t) \geq c_{MP}.
\end{equation}
We claim that 
\begin{equation}\label{5.7-}
    t_{\epsilon} \to 1,  \quad \text{ as }\,\, \epsilon \to 0.
\end{equation}
Indeed, from $\frac{d}{dt} I ( \gamma_{\epsilon}(t)) |_{t = t_{\epsilon}} = 0$ we obtain
\begin{equation}\label{5.7}
       t_{\epsilon}^{p-1}\int_{0}^{1}  r^{\alpha_{1}} | u_{\epsilon}^{\prime}|^{p} \, dr = t_{\epsilon}^{p^{*}-1}\int_{0}^{1} r^{\theta} |u_{\epsilon}|^{p^{*}} \left(\ln{(\tau + |t_{\epsilon} u_{\epsilon}|)}\right)^{r^{\beta}} \, dr.
\end{equation}
From  \eqref{3.6}, \eqref{3.7} and \eqref{5.7}, we can write
\begin{equation}\label{5.8}
\begin{aligned}
\mathcal{S}^{\frac{\theta+1}{\theta-\alpha_{1}+p}} + O(\epsilon^{sp}) &= t_{\epsilon}^{p^{*}-p}\int_{0}^{1} r^{\theta} |u_{\epsilon}|^{p^{*}} \left(\ln{(\tau + |t_{\epsilon} u_{\epsilon}|)}\right)^{r^{\beta}} \, dr \\
           &= t_{\epsilon}^{p^{*}-p}\int_{0}^{1} r^{\theta} |u_{\epsilon}|^{p^{*}} \, dr + t_{\epsilon}^{p^{*}-p} \mathcal{E}_{t_{\epsilon}}(0,1) \\
           &= t_{\epsilon}^{p^{*}-p}\Big[\mathcal{S}^{\frac{\theta+1}{\theta-\alpha_{1}+p}} + O(\epsilon^{sp^{*}}) + \mathcal{E}_{t_{\epsilon}}(0,1)\Big],
           \end{aligned}
\end{equation}
where $\mathcal{E}_{t_{\epsilon}}(0,1)$ is given by \eqref{EtAB}. 
By the mountain-pass geometry
structure in Lemma~\ref{geometria PM}, we have $\delta_1\le t_{\epsilon}\le  T$ for some $\delta
_1 > 0$.  So, we can suppose that $t_{\epsilon}\to t_0>0$ as $\epsilon\to 0$ and,  from the condition $0<\beta< \min\{(\theta+1)/p, sp\}$, Lemma~\ref{lema20} and Lemma~\ref{lema20-B}  we can write 
\begin{equation}\label{5.9}
    \mathcal{E}_{t_{\epsilon}}(0,1) = \left\{
\begin{array}{lll}
  O(\epsilon^{\beta} \ln{(|\ln {\epsilon}|)}) + O(\epsilon^{\frac{\theta+1}{p}}), & \hbox{if} & 1 \leq \tau < e \\
O(\epsilon^{\beta} \ln{(|\ln {\epsilon}|)}), & \hbox{if} & e \leq \tau < \infty
\end{array}
\right. = O(\epsilon^{\beta} \ln{(|\ln {\epsilon}|)}).
\end{equation}
From \eqref{5.8} and \eqref{5.9},  we get
\begin{equation}\label{OtO}
    1 + O(\epsilon^{sp}) 
   =  t_{\epsilon}^{p^{*}-p} ( 1 + O(\epsilon^{\frac{\theta+1}{p-1}}) + O(\epsilon^{\beta} \ln{(|\ln {\epsilon}|)})).
\end{equation}
Letting $\epsilon\to 0$ we obtain $t_{\epsilon}\to 1$ as claimed in \eqref{5.7-}. Directly from \eqref{OtO} and using the condition $0<\beta< \min\{(\theta+1)/p, sp\}$ again, we get
\begin{equation*} 
t_{\epsilon}^{p^*-p}=\frac{ 1 + O(\epsilon^{sp}) }{1 + O(\epsilon^{\frac{\theta+1}{p-1}}) + O(\epsilon^{\beta} \ln{(|\ln {\epsilon}|)})}=1+O(\epsilon^{\beta} \ln{(|\ln {\epsilon}|)}).
\end{equation*}
Now, for any $q>0$ we can write $(1+x)^{q}=1+qx+O(x^2)$ as $x\to 0$. Hence, from the above identity we can write
\begin{equation}\label{tpp*O}
\left\{\begin{aligned}
&t_{\epsilon}=1+O(\epsilon^{\beta} \ln{(|\ln {\epsilon}|)})\\
&t^{p}_{\epsilon}=1+pT_{\epsilon}+O(\epsilon^{2\beta} \ln^2{(|\ln {\epsilon}|)})\\
&t^{p^*}_{\epsilon}=1+p^*T_{\epsilon}+O(\epsilon^{2\beta} \ln^2{(|\ln {\epsilon}|)}),\\
\end{aligned}\right.
\end{equation}
where $T_{\epsilon}=t_{\epsilon}-1$. Now, from \eqref{definition u}, \eqref{5.3}, \eqref{5.7-}  we have
\begin{equation}\label{IG<}
\begin{aligned}
    \Big|\int_{0}^{1} r^{\theta}G(r, t_{\epsilon}u_{\epsilon})\,  dr\Big| &\leq   C \epsilon^{sp^{*}-\frac{np^{*}}{m}}\int_{0}^{\epsilon} r^{\theta+ \beta} \,  dr + C\epsilon^{sp^{*}}\int_{\epsilon}^{1} r^{\theta+ \beta- \frac{np^{*}}{m}} \,  dr\\
    &+  C\epsilon^{sp-\frac{np}{m}}\int_{0}^{\epsilon} r^{\theta+ \beta} \,  dr + C\epsilon^{sp}\int_{\epsilon}^{1} r^{\theta+ \beta- \frac{np}{m}} \,  dr\\
    &\leq C \left(\epsilon^{\beta} + \epsilon^{\frac{\theta+1}{p-1}}\right)+C \left(\epsilon^{ \beta+\theta-\alpha_{1}+p} + \epsilon^{\frac{\alpha_{1}-p+1}{p-1}}\right)\le C_1 \epsilon^{\beta},
    \end{aligned}
\end{equation}
for $\epsilon
>0$ small enough.  For $\tau\ge 1$, it follows from Lemma~\ref{lema20} that there exists $C>0$ such that  (recall $\beta<(\theta+1)/p$)
\begin{equation}\label{A>T}
 \mathcal{E}_{t_{\epsilon}}(0,1)\ge C \epsilon^{\beta} \ln{|\ln{\epsilon}|}
\end{equation}provided that $\epsilon>0$ is small enough. At this point, we are in a position to estimate the level $c_{MP}$. From \eqref{5.6} is it enough to show that
$$  I(t_{\epsilon} u_{\epsilon}) < \Big(\frac{1}{p}-\frac{1}{p^*}\Big)\mathcal{S}^{\frac{\theta+1}{\theta-\alpha_1+p}},$$
for $\epsilon>0$ small enough. To do this,  we first write 
\begin{equation}\label{I<cpn-P1}
\begin{aligned}
   I(t_{\epsilon} u_{\epsilon}) &= \frac{t_{\epsilon}^{p}}{p} \int_{0}^{1} r^{\alpha_{1}} |u_{\epsilon}'|^{p} \, dr - \dfrac{t_{\epsilon}^{p^{*}}}{p^{*}}\int_{0}^{1}  r^{\theta}  |u_{\epsilon}|^{p^{*}} (\ln{(\tau + t_{\epsilon} |u_{\epsilon}|)})^{r^{\beta}}\, dr + \int_{0}^{1} r^{\theta} G(r, t_{\epsilon}u_{\epsilon})\,  dr \\
 &= \frac{t_{\epsilon}^{p}}{p} \|u_{\epsilon}\|^{p} - \dfrac{t_{\epsilon}^{p^{*}}}{p^{*}}\Big(\|u_{\epsilon}\|^{p^{*}}_{L^{p^{*}}_{\theta}} +\mathcal{E}_{t_{\epsilon}}(0,1)\Big)+ \int_{0}^{1} r^{\theta} G(r, t_{\epsilon}u_{\epsilon})\,  dr.
\end{aligned}
\end{equation}
Combining \eqref{3.6}, \eqref{3.7},  \eqref{5.9}, \eqref{tpp*O},  \eqref{I<cpn-P1}, \eqref{IG<} and \eqref{A>T} we obtain
\begin{equation}
\begin{aligned}
    I(t_{\epsilon} u_{\epsilon})  
&=\Big( \frac{1}{p}+T_{\epsilon}+O(\epsilon^{2\beta} \ln^2{(|\ln {\epsilon}|)})\Big) \left(\mathcal{S}^{\frac{\theta+1}{\theta-\alpha_{1}+p}} + O(\epsilon^{sp})\right) \\
&- \Big( \frac{1}{p^*}+T_{\epsilon}+O(\epsilon^{2\beta} \ln^2{(|\ln {\epsilon}|)})\Big)\left(\mathcal{S}^{\frac{\theta+1}{\theta-\alpha_{1}+p}} + \mathcal{E}_{t_{\epsilon}}(0,1)+O(\epsilon^{sp^{*}})\right) +O(\epsilon^{\beta})\\
&\leq \left(\frac{1}{p} -\frac{1}{p^{*}}\right)\mathcal{S}^{\frac{\theta+1}{\theta-\alpha_{1}+p}} + \epsilon^{\beta}\ln{|\ln{\epsilon}|} \Big[-\frac{C}{p^*}+
 O\Big(\frac{\epsilon^{sp}}{\epsilon^{\beta}\ln{|\ln{\epsilon}|}}\Big) 
+O\Big(\epsilon^{\beta}\ln{|\ln{\epsilon}|}\Big)\Big] \\
&+ \epsilon^{\beta}\ln{|\ln{\epsilon}|} \Big[ O\Big(\frac{\epsilon^{sp^{*}}}{\epsilon^{\beta}\ln{|\ln{\epsilon}|} }\Big)+O\Big(\frac{\epsilon^{\beta}}{\epsilon^{\beta}\ln{|\ln{\epsilon}|}}\Big)\Big] \\
&< \Big(\frac{1}{p}-\frac{1}{p^*}\Big)\mathcal{S}^{\frac{\theta+1}{\theta-\alpha_1+p}},
\end{aligned}
\end{equation}
for $\epsilon>0$ sufficiently small.
\end{proof}
\begin{lemma}\label{lema A}
    The level $\Big(\frac{1}{p}-\frac{1}{p^*}\Big)\mathcal{S}^{\frac{\theta+1}{\theta-\alpha_1+p}}$ is non-compactness level for the functional $I$.
\end{lemma}
\begin{proof} 
Let $(u_{\epsilon})$ be defined as in \eqref{definition u} with $\widehat{A}=1$. We claim that $(u_{\epsilon})$ is a  concentrated sequence  at the origin $r=0$. Indeed, for $\rho<r_0$ and $r \in(0, \rho)$, we have $u_{\epsilon}(r)= \epsilon^{-(\alpha_{1}-p+1)/p} u_{1}^{*}(r/\epsilon)$  and,  from  \eqref{3.4} we have
\begin{equation*}
    \lim_{\epsilon\to 0}\int_{0}^{\rho} r^{\alpha_{1}} |u_{\epsilon}'(r)|^{p}\, dr 
    =  \lim_{\epsilon\to 0}\int_{0}^{\rho/\epsilon} s^{\alpha_{1}} |(u_{1}^{*})'(s)|^{p}\, ds = \int_{0}^{\infty} s^{\alpha_{1}} |(u_{1}^{*})'(s)|^{p}\, ds = \mathcal{S}^{\frac{\theta+1}{\theta-\alpha_{1}+p}}.
\end{equation*}
 Thus,  \eqref{3.6} ensures
\begin{eqnarray*}
   \int_{\rho}^{1} r^{\alpha_{1}} |u_{\epsilon}'(r)|^{p}\, dr &=&       \int_{0}^{1} r^{\alpha_{1}} |u_{\epsilon}'(r)|^{p}\, dr  -       \int_{0}^{\rho} r^{\alpha_{1}} |u_{\epsilon}'(r)|^{p}\, dr\\ 
      &=& \mathcal{S}^{\frac{\theta+1}{\theta-\alpha_{1}+p}} + O(\epsilon^{sp})- \mathcal{S}^{\frac{\theta+1}{\theta-\alpha_{1}+p}}\\
      &\to& 0
\end{eqnarray*}
as  $\epsilon \to 0$. If $\rho \in [r_{0}, 1)$, it follows that
$$    \int_{\rho}^{1} r^{\alpha_{1}} |u_{\epsilon}'(r)|^{p}\, dr  \leq     \int_{r_{0}/2}^{1} r^{\alpha_{1}} |u_{\epsilon}'(r)|^{p}\, dr  \to 0 \quad \text{if}\,\, \epsilon \to 0.$$
Hence, for any $0<\rho<1$ we have 
\begin{equation}\label{Diric-0}
\lim_{\epsilon\to 0}  \int_{\rho}^{1} r^{\alpha_{1}} |u_{\epsilon}'(r)|^{p}\, dr=0.
\end{equation}
In addition, we have 
\begin{eqnarray*}
    \int_{0}^{1} r^{\theta} |u_{\epsilon}|^{p}\, dr &\leq & \int_{0}^{1} r^{\theta} |u^{*}_{\epsilon}|^{p}\, dr \\
 &=& \epsilon^{  \theta -\alpha_{1}+p}\int_{0}^{\epsilon^{-1}} s^{\theta} |u^{*}_{1}(s)|^{p}\, ds \\
&=& \widehat{c}^{p} \epsilon^{  \theta -\alpha_{1}+p}\int_{0}^{\epsilon^{-1}} \dfrac{s^{\theta}}{(1+s^{n})^{\frac{1}{m}}}\, ds \to 0, \quad \text{as}\,\, \epsilon \to 0\\
\end{eqnarray*}
where $m$ and $n$ are gived in \eqref{relations-paramenters}.  In view of the compact embedding \eqref{eq10}, we have $u_{\epsilon}\rightharpoonup 0$  weakly in $X^{1,p}_{1}(\alpha_0,\alpha_1)$ as $\epsilon \to 0$. This together with \eqref{Diric-0} ensures that $(u_\epsilon)$ is a concentrated sequence at the origin $r=0$.  Now, by using \eqref{3.6}, \eqref{3.7}, \eqref{5.9} and \eqref{IG<}  we obtain
\begin{eqnarray*}
 I(u_{\epsilon}) 
& =& \Big(\frac{1}{p}-\frac{1}{p^{*}}\Big)\mathcal{S}^{\frac{\theta+1}{\theta-\alpha_{1}+p}} + O(\epsilon^{sp})- O(\epsilon^{sp^{*}})+ \dfrac{\mathcal{E}_{1}(0,1)}{p^{*}}+ \int_{0}^{1} r^{\theta} G(r, u_{\epsilon})\,  dr\\
&\to & \Big(\frac{1}{p}-\frac{1}{p^*}\Big)\mathcal{S}^{\frac{\theta+1}{\theta-\alpha_1+p}}
\end{eqnarray*}
as  $\epsilon \to 0$. Hence,  $(u_{\epsilon})$ is concentrating and converges weakly to $0$, and thus it does not contain
a strongly convergent subsequence in  $X^{1,p}_{1}(\alpha_0,\alpha_1)$.
\end{proof}

In view of Lemmas  \ref{geometria PM} and \ref{lema B} we are in a position to apply \cite[Teorema 2.2]{MR0709644}, to get a Palais-Smale sequence $(u_{j})$ in $ X^{1,p}_{1}(\alpha_0,\alpha_1)$ of $I$ at level $c_{MP} < \big(\frac{1}{p} -\frac{1}{p^{*}}\big)\mathcal{S}^{\frac{\theta+1}{\theta-\alpha_{1}+p}}$. That is, for any $\varphi \in X^{1, p}_{1}(\alpha_0,\alpha_1)$ holds
\begin{equation}\label{cond PS}
    I(u_{j}) \to c_{MP}\quad \text{and}\quad \langle I'(u_{j}), \varphi \rangle \to 0\;\;\mbox{as}\;\; j\to \infty.
\end{equation}
\begin{lemma}\label{seq limitada}
Up to a subsequence, $(u_{j})$ converges weakly to  $u_0\in X^{1,p}_{1}(\alpha_0,\alpha_1)$. In addition, $u_0$ is a nontrivial weak solution to     \eqref{C3-T4}.
\end{lemma}
\begin{proof}
First,  from \eqref{cond PS}  we obtain 
\begin{eqnarray*}
c_{MP}+1 &\geq& I(u_{j}) - \frac{1}{p^{*}}\langle I'(u_{j}), u_{j} \rangle \\
&=&      \Big (\frac{1}{p}-\frac{1}{p^{*}}\Big)\|u_j\|^{p}  + \int_{0}^{1} r^{\theta} G(r, u_{j})\,  dr.
\end{eqnarray*}
It follows that $(u_{j})$ is a bounded sequence in  $X^{1,p}_{1}(\alpha_0,\alpha_1)$. Hence,  up to a subsequence, there exists  $u_{0}$ in $ X^{1,p}_{1}(\alpha_0,\alpha_1)$ such that
\begin{equation}\label{3-conv}
u_{j} \rightharpoonup u_{0}\;\;\mbox{weakly in}\;\; X^{1,p}_{1}(\alpha_0,\alpha_1),\;\; u_{j} \to u_{0} \;\;\mbox{in}\;\; L^{q}_{\theta}\; (p\leq q < p^{*}) \;\; \mbox{and}\;\; u_{j} \to u_{0}\;\;\mbox{a.e in }\;\; (0,1).
\end{equation}
By a standard argument, we can verify that  $u_0$ solves  the equation \eqref{C3-T4}, see for instance \cite{CCM02017,deFigueiredo-Goncalves-Miyagaki} for more details. Thus, it is sufficient to prove the following:
\begin{claim}\label{u0-nontrivial}
 $u_{0} \not \equiv 0.$
\end{claim}
 By contradiction, we suppose that $u_{0}\equiv 0$. As well as in \eqref{4.7-} we have
        \begin{equation}\label{5.12}
            \int_{\rho}^{1} r^{\theta} |u_{j}|^{p^{*}} (\ln{(\tau + |u_{j}|)})^{r^{\beta}}\, dr \to 0, 
        \end{equation}
        for any $\rho\in (0,1)$ fixed.
 In addition, the pointwise convergence in \eqref{3-conv}, Lemma~\ref{C3-L1} and the dominated convergence theorem implies
  \begin{equation}\label{5.12-part2}
            \int_{\rho}^{1} r^{\theta} |u_{j}|^{p^{*}}\, dr \to 0.
        \end{equation}
 Let $\Bar{\eta}:[0,1]\to[0,1]$ be a smooth function such that  $\Bar{\eta} \equiv 0$ in $[0, \rho/2]$ and $\Bar{\eta} \equiv 1$ in $[\rho, 1]$. By choosing $\varphi = \Bar{\eta} u_{j}$ in \eqref{cond PS}  we obtain
\begin{eqnarray*}
   \int_{\frac{\rho}{2}}^{1} r^{\alpha_{1}} |u^{\prime}_j|^{p-2}u'_{j}(\eta u_{j})' \, dr &=&    \int_{\frac{\rho}{2}}^{1} r^{\theta} |u'_{j}|^{p^{*}}\Bar{\eta} (\ln{(\tau+|u_{j}|)})^{r^{\beta}} \, dr + \langle o_{j}(1), \Bar{\eta} u_{j} \rangle \to 0,
\end{eqnarray*}
as $j \to \infty$, which implies that
\begin{eqnarray*}
   \int_{\rho}^{1} r^{\alpha_{1}} |u'_{j}|^{p} \, dr  \to 0, \quad \text{for}\,\, \rho \in (0,1).
\end{eqnarray*}
We claim that
\begin{equation}\label{5.13}
    I(u_{j}) = I_{0}(u_{j})+ o_{j}(1),
\end{equation}
where
\begin{equation*}
    I_{0}(u)= \frac{1}{p} \int_{0}^{1} r^{\alpha_{1}} |u'|^{p} \,dr - \frac{1}{p^{*}} \int_{0}^{1} r^{\theta} |u|^{p^{*}} \,dr. 
\end{equation*}
Indeed, note that
\begin{eqnarray*}
    I(u_{j}) - I_{0}(u_{j}) &=& -\frac{1}{p^{*}} \int_{0}^{1}  r^{\theta} |u_{j}|^{p^{*}} \left((\ln{(\tau + |u_{j}|)})^{r^{\beta}}-1\right)\, dr + \int_{0}^{1} r^{\theta} G(r, u_{j})  \,dr. 
\end{eqnarray*}
We will estimate each integral above for $j$ large enough.  By using \eqref{5.12} and \eqref{5.12-part2} we can see
\begin{equation}\label{UU}
\begin{aligned}
\Big| \int_{0}^{1}  r^{\theta} |u_{j}|^{p^{*}} \left((\ln{(\tau + |u_{j}|)})^{r^{\beta}}-1\right)\, dr\Big| &\le  \Big|\int_{0}^{\rho}  r^{\theta} |u_{j}|^{p^{*}} \left((\ln{(\tau + |u_{j}|)})^{r^{\beta}}-1\right)\, dr\Big|\\
&+\Big| \int_{\rho}^{1}  r^{\theta} |u_{j}|^{p^{*}} \left((\ln{(\tau + |u_{j}|)})^{r^{\beta}}-1\right)\, dr\Big|\\
&\le  \Big|\int_{0}^{\rho}  r^{\theta} |u_{j}|^{p^{*}} \left((\ln{(\tau + |u_{j}|)})^{r^{\beta}}-1\right)\, dr\Big|+o_{j}(1).
\end{aligned}
\end{equation}
For $\tau\ge e$,  arguing as in \eqref{neweq1}  we have 
\begin{equation}\label{U_1}
\begin{aligned}
\Big|\int_{0}^{\rho}  r^{\theta} |u_{j}|^{p^{*}} \left((\ln{(\tau + |u_{j}|)})^{r^{\beta}}-1\right)\, dr\Big|& = \int_{0}^{\rho}  r^{\theta} |u_{j}|^{p^{*}} \left((\ln{(\tau + |u_{j}|)})^{r^{\beta}}-1\right)\, dr \le 
o_{\rho}(1)
\end{aligned}
\end{equation}
where $o_{\rho}(1)\to 0$ as $\rho\to 0$ uniformly on $j$. On the other hand, if $1\le \tau <e$, by using the $u_0\equiv 0$ and \eqref{3-conv} and taking $j$ large enough. we can assume  $|u_j|\le e-\tau$ on $(0,\rho)$. Hence 
\begin{equation}\label{U-2}
\begin{aligned}
\Big|\int_{0}^{\rho}  r^{\theta} |u_{j}|^{p^{*}} \left((\ln{(\tau + |u_{j}|)})^{r^{\beta}}-1\right)\, dr\Big|&=\int_{0}^{\rho} r^{\theta} |u_{j}|^{p^{*}} \left(1-(\ln{(\tau + |u_{j}|)})^{r^{\beta}}\right)\, dr\\
& \le (e-\tau)^{p^{*}} \int_{0}^{\rho} r^{\theta} \, dr=o_{\rho}(1).
\end{aligned}
\end{equation}
 Now, from \eqref{5.3} and \eqref{3-conv}  we obtain
\begin{equation}\label{U-G}
\begin{aligned}
        \int_{0}^{1} r^{\theta} |G(r, u_{j})| \, dr &\leq  \epsilon \int_{0}^{1} r^{\theta}  |u_{j}|^{p^{*}} \, dr+ C_{\epsilon} \int_{0}^{1} r^{\theta} |u_{j}|^{p} \, dr\\
    &\leq \epsilon \Sigma_{p} + C_{\epsilon} o_j(1).
    \end{aligned}
\end{equation} 
By combining \eqref{UU}, \eqref{U_1}, \eqref{U-2} and \eqref{U-G}, we can write 
\begin{equation}\nonumber
 |I(u_{j}) - I_{0}(u_{j})| = o_{\rho}(1)+o_{j}(1)+o_{\epsilon}(1)
\end{equation}
which proves  \eqref{5.13}. Analogously, one shows that
$$\langle I'(u_{j}), \varphi \rangle = \langle I_{0}'(u_{j}), \varphi \rangle + \langle o_{j}(1), \varphi \rangle, \quad \forall \,\, \varphi \in X^{1,p}_{1}(\alpha_0,\alpha_1).$$
Then, using \eqref{cond PS} we obtain $(u_{j})$ is a Palais-Smale sequence also to the functional $I_{0}$ at level $c_{MP}$.  Namely, for any $\varphi \in X^{1, p}_{1}(\alpha_0,\alpha_1)$ we have
\begin{equation}\label{cond PS-I0}
    I_0(u_{j}) \to c_{MP}\quad \text{and}\quad \langle I'_0(u_{j}), \varphi \rangle \to 0\;\;\mbox{as}\;\; j\to \infty.
\end{equation}
In particular, 
\begin{equation}\label{WS1}
\|u_j\|^{p}=\|u_j\|^{p^*}_{L^{p^*}_{\theta}}+ \langle I^{\prime}_{0}(u_{j}), u_{j} \rangle
\end{equation}
and 
\begin{equation}\label{w1}
    I_{0}(u_{j}) - \frac{1}{p^{*}}\langle I^{\prime}_{0}(u_{j}), u_{j} \rangle = \Big(\frac{1}{p}-\frac{1}{p^{*}}\Big)\|u_{j}\|^{p}, \quad \mbox{for all}\;\; j\in\mathbb{N}.
\end{equation}
From \eqref{WS1}, we can suppose that  there exists $b\ge 0$ such that 
$$0\le b=\lim_{j\to\infty}\|u_{j}\|^{p}=\lim_{j\to\infty}\|u_j\|^{p^*}_{L^{p^*}_{\theta}}.$$
By \eqref{eq10}, we also have
$$\|u_{j}\|^{p}_{L^{p^{*}}_{\theta}} \leq \mathcal{S}^{-1}\|u_{j}\|^{p}$$
which implies 
\begin{equation}\label{b-Best}
b^{\frac{p}{p^*}}\leq \mathcal{S}^{-1}b.
\end{equation}
But, letting $j\to\infty$ in \eqref{w1}, we get 
\begin{equation}\label{w111}
    c_{MP} = \Big(\frac{1}{p}-\frac{1}{p^{*}}\Big)b.
\end{equation}
Combining the condition $c_{MP} < \big(\frac{1}{p} -\frac{1}{p^{*}}\big)\mathcal{S}^{\frac{\theta+1}{\theta-\alpha_{1}+p}}$ with \eqref{b-Best} and \eqref{w111}  we obtain $b=0$.  Hence, up to a subsequence, we have $u_{j} \to 0$ strongly in $X^{1,p}_{1}(\alpha_0,\alpha_1)$ and thus $I(u_{j}) \to 0$ which contradicts $I(u_{j}) \to c_{MP}>0$. 
\end{proof}

\end{document}